\documentclass[11pt]{amsart}
\usepackage[dvipsnames,usenames]{xcolor}
\usepackage[colorlinks=true, urlcolor=NavyBlue, linkcolor=NavyBlue, citecolor=NavyBlue]{hyperref}
\usepackage{graphicx}
\usepackage{epsfig}
\usepackage[latin1]{inputenc}
\usepackage{amsmath}
\usepackage{amsfonts}
\usepackage{amssymb}
\usepackage{amsthm}
\usepackage{amscd}
\usepackage{verbatim}
\usepackage{subfigure}
\usepackage{pinlabel}
\usepackage{enumerate, enumitem}
\usepackage{tikz,tikz-3dplot}

\usepackage{tikz}
\usetikzlibrary{arrows}
\usetikzlibrary{decorations.pathreplacing}
\usepackage{verbatim}
	
\newtheorem{theorem}{Theorem}[section]
\newtheorem{Theorem}{Theorem}
\newtheorem{lemma}[theorem]{Lemma}
\newtheorem{proposition}[theorem]{Proposition}

\newtheorem{corollary}[theorem]{Corollary}

\theoremstyle{definition}
\newtheorem{definition}[theorem]{Definition}

\theoremstyle{remark}
\newtheorem{remark}[theorem]{Remark}
\newtheorem{example}[theorem]{Example}

\def\t{\mathbf{t}}

\def\Z{\mathbb{Z}}

\def\F{\mathbb{F}}
\def\Q{\mathbb{Q}}

\def\cA{\mathcal{A}}
\def\cAr{\mathcal{A}^{\textup{red}}}
\def\cE{\mathcal{E}}
\def\cEr{\mathcal{E}^{\textup{red}}}

\def\cH{\mathcal{H}}
\def\cI{\mathcal{I}}
\def\cC{\mathcal{C}}

\def\HH{\mathbb{H}}
\def\cI{\mathcal{I}}

\def\lk{\textup{lk}}

\def\min{\textup{min}}

\def\aa{\mathbf{a}}
\def\hh{\mathbf{h}}
\def\cM{\mathcal{M}}
\def\cT{\mathcal{T}}
\def\sce{\textsc{e}}

\def\sct{\textsc{t}}
\def\bell{\boldsymbol{\ell}}

\DeclareMathOperator{\HFL}{HFL^{--}}
\DeclareMathOperator{\HFK}{HFK^{--}}
\DeclareMathOperator{\Ker}{Ker}
\DeclareMathOperator{\HFhat}{\widehat{HF}}

\DeclareMathOperator{\HFKhat}{\widehat{HFK}}
\DeclareMathOperator{\HFLhat}{\widehat{HFL}}
\DeclareMathOperator{\bideg}{bideg}
\DeclareMathOperator{\CFD}{\widehat{CFD}}

\title{Cable links and L-space surgeries}

\author[Eugene Gorsky]{Eugene Gorsky}
\thanks{The first author was partially supported by RFBR grant 13-01-00755 and NSF grant DMS-1403560.}
\address {Department of Mathematics, Columbia University, 2990 Broadway \\ New York, NY 10027 }
\address {Department of Mathematics, UC Davis, One Shields Ave\\ Davis, CA 95616}
\address {International Laboratory of Representation Theory and Mathematical Physics\\
NRU-HSE, 7 Vavilova St.\\ 
Moscow, Russia 117312}
\email{egorsky@math.columbia.edu}

\author[Jennifer Hom]{Jennifer Hom}
\thanks{The second author was partially supported by NSF grant DMS-1307879.}
\address {Department of Mathematics, Columbia University, 2990 Broadway \\ New York, NY 10027}
\email{hom@math.columbia.edu}

\numberwithin{equation}{section}

\begin{document}

\begin{abstract}
An L-space link is a link in $S^3$ on which all sufficiently large integral surgeries are L-spaces. We prove that for $m, n$ relatively prime, the $r$-component cable link $K_{rm,rn}$ is an L-space link if and only if $K$ is an L-space knot and $n/m \geq 2g(K)-1$. We also compute $\HFL$ and $\HFLhat$ of an L-space cable link in terms of its Alexander polynomial. As an application, we confirm a conjecture of Licata \cite{Licata} regarding the structure of $\HFLhat$ for $(n,n)$ torus links.
\end{abstract}

\maketitle

\section{Introduction}
Heegaard Floer homology is a package of $3$-manifold invariants defined by Ozsv\'ath and Szab\'o \cite{OSproperties, OS3mfds}. In its simplest form, it associates to a closed $3$-manifold $Y$ a graded vector space $\HFhat(Y)$. For a rational homology sphere $Y$, they show that
\[ \dim \HFhat(Y) \geq |H_1(Y; \Z)|. \]
If equality is achieved, then $Y$ is called an \emph{L-space}. 

A knot $K \subset S^3$ is an \emph{L-space knot} if $K$ admits a positive L-space surgery. Let $S^3_{p/q}(K)$ denote $p/q$ Dehn surgery along $K$. If $K$ is an L-space knot, then $S^3_{p/q}(K)$ is an L-space for all $p/q \geq 2g(K)-1$, where $g(K)$ denotes the Seifert genus of $K$ \cite[Corollary 1.4]{OSrational}.
A link $L \subset S^3$ is an \emph{L-space link} if all sufficiently large integral surgeries on $L$ are L-spaces. In contrast to the knot case, if $L$ admits a positive L-space surgery, it does not necessarily follow that all sufficiently large surgeries are also L-spaces; see \cite[Example 2.3]{Liu}.

For relatively prime integers $m$ and $n$, let $K_{m,n}$ denote the $(m, n)$ cable of $K$, where $m$ denotes the longitudinal winding. Without loss of generality, we will assume that $m>0$. Work of Hedden \cite{HeddencablingII} (``if'' direction) and the second author \cite{Homcabling} (``only if'' direction) completely classifies L-space cable knots.

\begin{Theorem}[\cite{HeddencablingII,Homcabling}]
\label{thm:cableknot}
Let $K$ be a knot in $S^3$, $m>1$ and $\gcd(m,n)=1$. The cable knot $K_{m, n}$ is an L-space knot if and only if $K$ is an L-space knot and $n/m> 2g(K)-1$.
\end{Theorem}

\begin{remark}
Note that when $m=1$, we have that $K_{1,n}=K$ for all $n$.
\end{remark}

\noindent We generalize this theorem to cable links with many components. Throughout the paper, we assume that each component of a cable link is oriented in the same direction.

\begin{Theorem}\label{thm:cablelink}
Let $K$ be a knot in $S^3$ and $\gcd(m,n)=1$. The $r$-component cable link $K_{rm, rn}$ is an L-space link if and only if $K$ is an L-space knot and $n/m \geq 2g(K)-1$.
\end{Theorem}

\noindent In \cite{OSlens}, Ozsv\'ath and Szab\'o show that if $K$ is an L-space knot, then $\HFKhat(K)$ is completely determined by $\Delta_K(t)$, the Alexander polynomial of $K$. Consequently, the Alexander polynomials of L-space knots are quite constrained (the non-zero coefficients are all $\pm 1$ and alternate in sign) and the rank of $\HFKhat(K)$ is at most one in each Alexander grading. In \cite[Theorem 1.15]{Liu}, Liu generalizes this result to give bounds on the rank of $\HFL(L)$ in each Alexander multi-grading and on the coefficients of the multi-variable Alexander polynomial of an L-space link $L$ in terms of the number of components of $L$. For L-space cable links, we have the following stronger result.

\begin{definition}
Define the $\Z$-valued functions $\hh(k)$ and $\beta(k)$ by the equations:
\begin{equation}
\label{def hh}
\sum_{k}\hh(k)t^k=\frac{t^{-1}\Delta_{m,n}(t)(t^{mnr/2}-t^{-mnr/2})}{(1-t^{-1})^2(t^{mn/2}-t^{-mn/2})},\qquad \beta(k)=\hh(k-1)-\hh(k)-1,
\end{equation}
where $\Delta_{m,n}(t)$ is the Alexander polynomial of the cable knot $K_{m,n}$.
\end{definition}

Throughout, we work with $\F=\Z/2\Z$ coefficients. The following theorem gives a complete description of the homology groups $\HFLhat$ for cable links with $n/m>2g(K)-1$.

\begin{Theorem}
\label{hat homology}
Let $K_{rm,rn}$ be a cable link with $n/m>2g(K)-1$.
\begin{enumerate}[label=(\alph*)]
\item If $\beta(k)+\beta(k+1)\le r-2$ then: 
$$
\HFLhat(K_{rm,rn},k,\ldots,k)\simeq \bigoplus_{i=0}^{\beta(k)}\binom{r-1}{i}\F_{-2\hh(k)-i}\oplus
\bigoplus_{i=0}^{\beta(k+1)}\binom{r-1}{i}\F_{-2\hh(k)+2-r+i} 
$$
\item If $\beta(k)+\beta(k+1)\ge r-2$ then:
$$
\HFLhat(K_{rm,rn},k,\ldots,k)\simeq \bigoplus_{i=0}^{r-2-\beta(k+1)}\binom{r-1}{i}\F_{-2\hh(k)-i}\oplus
\bigoplus_{i=0}^{r-2-\beta(k)}\binom{r-1}{i}\F_{-2\hh(k)+2-r+i}
$$
\item If $v$ has $j$ coordinates equal to $k-1$ and $r-j$ coordinates equal to $k$ for some $k$ and $1\le j\le r-1$, then:
$$\HFLhat(K_{rm,rn},(k-1)^{j},k^{r-j})\simeq \binom{r-2}{\beta(k)}\F_{-2\hh(k)-\beta(k)-j}.$$
\item For all other Alexander gradings the groups $\HFLhat$  vanish.
\end{enumerate}
\end{Theorem}

\noindent We prove the parts of this theorem as separate Theorems \ref{diagonal degeneration}, \ref{diagonal dual} and \ref{off diagonal}. We compute $\HFLhat$ explicitly for several examples in Section \ref{sec:examples}. In particular, we use Theorem \ref{hat homology} to confirm a conjecture of Joan Licata \cite[Conjecture 1]{Licata} concerning $\HFLhat$ for $(n,n)$ torus links.

\begin{Theorem}
\label{th: n n hat}
Suppose that $0\le s\le \frac{n-1}{2}$. Then 
$$
\HFLhat\left(T(n,n),\frac{n-1}{2}-s,\ldots,\frac{n-1}{2}-s\right)=\bigoplus_{i=0}^{s} \binom{n-1}{i}\F_{(-s^2-s-i)}\oplus \bigoplus_{i=0}^{s-1} \binom{n-1}{i}\F_{(-s^2-s-n+2+i)}.
$$
\end{Theorem} 

\noindent Combined with \cite[Theorem 2]{Licata}, this completes the description of $\HFLhat(T(n,n))$.

The following theorem describes the homology groups $\HFL$ for cable links with $n/m>2g(K)-1$.

\begin{Theorem}
\label{homology}
Let $K$ be an L-space knot and $n/m>2g(K)-1$. Consider an Alexander grading $v=(v_1,\ldots,v_n)$. Suppose that among the coordinates $v_i$ exactly $\lambda$ are equal to $k$ and all other coordinates are less than $k$. Let $|v|=v_1+\ldots+v_n$. Then the Heegaard-Floer homology group $\HFL(K_{rm,rn},v)$ can be described as follows:
\begin{enumerate}[label=(\alph*)]
\item \label{it:HFL0} If $\beta(k)<r-\lambda$ then $\HFL(K_{rm,rn},v)=0$.
\item \label{it:HFLF} If $\beta(k)\ge r-\lambda$ then
$$\HFL(K_{rm,rn},v)\simeq (\F_{(0)}\oplus \F_{(-1)})^{r-\lambda}\otimes \bigoplus_{i=0}^{\beta(k)-r+\lambda} \binom{\lambda-1}{i}\F_{(-2h(v)-i)},$$
\end{enumerate}
where
$h(v)=\hh(k)+kr-|v|$.
\end{Theorem}

\noindent We prove this theorem in Section \ref{sec:minus}.
The structure of the homology for $n/m=2g(K)-1$ (which is possible only if $m=1$) is more subtle and is described in Theorem \ref{homology special}.

Finally, we describe $\HFL$ as an $\F[U_1,\ldots,U_r]$--module. We define a collection of $\F[U_1,\ldots,U_r]$--modules $M_{\beta}$ for $0\le \beta\le r-2$, $M_{r-1,k}$ for $k\ge 0$ and $M_{r-1,\infty}$. These modules can be defined combinatorially and do not depend on a link.

\begin{Theorem}
\label{th:splitting} 
Let $R=\F[U_1,\ldots,U_r]$ and suppose that $n/m>2g(K)-1$. There exists a finite collection of diagonal lattice points $\aa_i=(a_i,\ldots,a_i)$ (determined by $m,n$ and the Alexander polynomial of $K$) such that $\HFL$ admits the following direct sum decomposition:
$$
\HFL(K_{rm,rn})=\bigoplus_i R\cdot \HFL(K_{rm,rn},\aa_i).
$$
Furthermore, for $\beta(a_i)\le r-2$ one has  $R\cdot \HFL(K_{rm,rn},\aa_i)\simeq M_{\beta(a_i)}$, 
and for $\beta(a_i)=r-1$ one has either $R\cdot \HFL(K_{rm,rn},\aa_i)\simeq M_{r-1,k}$ for some $k$
or $R\cdot \HFL(K_{rm,rn},\aa_i)\simeq M_{r-1,\infty}$.
\end{Theorem}

\noindent We compute $\HFL$ explicitly for several examples in Section \ref{sec:examples}.



\section*{Acknowledgments}

We are grateful to Jonathan Hanselman, Matt Hedden, Yajing Liu, Joan Licata, and Andr\'as N\'emethi for useful discussions. 

\section{Dehn surgery and cable links}
In this section, we prove Theorem \ref{thm:cablelink}. We begin with a result about Dehn surgery on cable links (cf. \cite{Heil}).

\begin{proposition}
\label{surgery}
The manifold obtained by $(mn, p_2, \dots, p_r)$--surgery on the $r$-component link $K_{rm, rn}$ is homeomorphic to $S^3_{n/m}(K) \# L(m, n) \# L(p_2-mn, 1) \# \dots \# L(p_r-mn, 1)$.
\end{proposition}

\begin{proof}
Recall (see, for example, \cite[Section 2.4]{HeddencablingII}) that $mn$-surgery on $K_{m,n}$ gives the manifold $S^3_{n/m}(K) \# L(m,n)$. Viewing $K_{m,n}$ as the image of $T_{m,n}$ on $\partial N(K)$, we have that the reducing sphere is given by the annulus $\partial N(K) \setminus N(T_{m,n})$ union two parallel copies of the meridional disk of the surgery solid torus; we obtain a sphere since the surgery slope coincides with the surface framing.

The link $K_{rm, rn}$ consists of $r$ parallel copies of $K_{m,n}$ on $\partial N(K)$. Label these $r$ copies $K^1_{m,n}$ through $K^r_{m,n}$. We perform $mn$-surgery on $K^1_{m,n}$ and consider the image $\widetilde{K}^i_{m,n}$ of $K^i_{m,n}$, $2 \leq i \leq r$, in $S^3_{n/m}(K) \# L(m,n)$. Each $\widetilde K^i_{m,n}$ lies on $\partial N(K) \setminus N(T_{m,n})$ and thus on the reducing sphere. In particular, each $\widetilde K^i_{m,n}$ bounds a disk $D^2_i$ in $S^3_{n/m}(K) \# L(m,n)$ such that the collection $\{D^2_2, \dots, D^2_r\}$ is disjoint. It follows that performing surgery on $\bigcup_{i=2}^r \widetilde K^i_{m,n} $ yields $r-1$ lens space summands. To see which lens spaces we obtain, note that the $mn$-framed longitude on $K^i_{m,n} \subset S^3$ coincides with the $0$-framed longitude on $\widetilde K^i_{m,n} \subset S^3_{n/m}(K) \# L(m,n)$. Thus, $p_i$-surgery on $K^i_{m,n}$ corresponds to $(p_i-mn)$-surgery on $\widetilde K^i_{m,n}$, and the result follows.
\end{proof}

Let us recall that the linking number between each two components of $K_{rm, rn}$ equals $l:=mn$. It is well-known that the cardinality of $H_1$ of the
 manifold obtained by $(p_1, p_2, \dots, p_r)$-surgery on $K_{rm, rn}$ equals $|\det \Lambda(p_1,\ldots,p_r)|$, where
 $$
 \Lambda_{ij}=\begin{cases}
p_i,& \text{if}\ i=j,\\
l, & \text{if}\ i\neq j.
 \end{cases}
$$ 
This cardinality can be computed using the following result.

\begin{proposition}
One has the following identity:
\begin{equation}
\label{det lambda}
\det \Lambda(p_1,\ldots,p_r)=(p_1-l)\cdots(p_r-l)+l\sum_{i=1}^{r}(p_1-l)\cdots\widehat{(p_i-l)}\cdots (p_r-l).
\end{equation}
\end{proposition}
\begin{proof}
One can easily check that $\det \Lambda(l,p_2,\ldots,p_r)=l(p_2-l)\cdots(p_r-l).$
The expansion of the determinant in the first row yields a recursion relation
$$
\det \Lambda(p_1,\ldots,p_r)=\det \Lambda(l,p_2\ldots,p_r)+(p_1-l)\det \Lambda(p_2,\ldots,p_r)=
$$
$$
=l(p_2-l)\cdots(p_r-l)+(p_1-l)\det \Lambda(p_2,\ldots,p_r).
$$
Now \eqref{det lambda} follows by induction in $r$.
\end{proof}
\begin{corollary}
\label{positive definite}
If $p_i\ge l$ for all $i$ then $\det \Lambda(p_1,\ldots,p_r)\ge 0$. 
\end{corollary}

In order to prove Theorem \ref{thm:cablelink}, we will need the following:
\begin{theorem}[{\cite[Proposition 1.11]{Liu}}]
\label{L space criterion}
A link $L$ is an $L$--space link if and only if there exists a surgery framing $\Lambda(p_1,\ldots,p_r)$, such that for all sublinks $L'\subseteq L$, 
$\det(\Lambda(p_1,\ldots, p_r)|_{L'}) >0$ and $S^3_{\Lambda|_{L'}}(L')$ is an $L$--space.
\end{theorem}

\noindent We will also need the following proposition, which we prove in Subsection \ref{subsec:bordered} below.

\begin{proposition}\label{prop:nonLspace}
Let $K$ be an L-space knot and $p_i > 0$, $i=1, \dots, r$. If $n < 2g(K)-1$, then the manifold obtained by $(p_1, \dots, p_r)$-surgery on the $r$-component link $K_{r, rn}$ is not an L-space.
\end{proposition}

\begin{proof}[Proof of Theorem \ref{thm:cablelink}.]
If $K_{rm, rn}$ is an L-space link, then by \cite[Lemma 1.10]{Liu} all its components are L-space knots. On the other hand, its components are isotopic to $K_{m, n}$. Thus, if $m>1$, then by Theorem \ref{thm:cableknot}, $K$ is an L-space knot and $n/m > 2g(K)-1$. If $m=1$, then $K$ must be an L-space knot and by Proposition \ref{prop:nonLspace}, $n \geq 2g(K)-1$.

Conversely, suppose that $K$ is an L-space knot and $n/m \geq 2g(K)-1$, i.e., $K_{m,n}$ is an L-space knot. Let us prove by induction on $r$ that $(p_1,\ldots,p_r)$-surgery on $K_{rm, rn}$ is an L-space if   $p_i>l$ for all $i$. For $r=1$ it is clear. By Proposition \ref{surgery}, the link $K_{rm, rn}$ admits an L-space surgery with parameters $l,p_2,\ldots,p_r$. Let us  apply Theorem \ref{L space criterion}. Indeed, by Corollary \ref{positive definite}, one has $\det(\Lambda(l,p_2\ldots, p_r)|_{L'}) >0$ and by the induction assumption
$S^3_{\Lambda(l,p_2\ldots, p_r)|_{L'}}(L')$ is an L--space for all sublinks $L'$. By \cite[Lemma 2.5]{Liu}, $(p_1,\ldots,p_r)$-surgery on $K_{rm, rn}$ is also an L-space for all $p_1>l$. Therefore $K_{rm, rn}$ is an L-space link.
\end{proof}

\subsection{Proof of Proposition \ref{prop:nonLspace}}\label{subsec:bordered}

We will prove Proposition \ref{prop:nonLspace} using Lipshitz-Ozsv\'ath-Thurston's bordered Floer homology \cite{LOT}, specifically Hanselman-Watson's \cite{HanselmanWatson} loop calculus. That is, we will decompose the result of surgery on $K_{r, rn}$ into two pieces, one that is surgery on a torus link in the solid torus and the other the knot complement, and then apply a gluing result of Hanselman-Watson to conclude that the result of this surgery along $K_{r, rn}$ is not an L-space. The following was described to us by Jonathan Hanselman.

Let $Y_1$ denote the Seifert fibered space obtained by performing $(p_1, \dots, p_r)$-surgery on the $r$-component $(r, 0)$-torus link in the solid torus. Consider the bordered manifold $(Y_1, \alpha_1, \beta_1)$, where $\alpha_1$ is the fiber slope and $\beta_1$ lies in the base orbifold; that is, $\alpha_1$ is the longitude and $\beta_1$ the meridian of the original solid torus. Let $(Y_2, \alpha_2, \beta_2)$ be the $n$-framed complement of $K$; that is, $Y_2 = S^3 \setminus N(K)$, $\alpha_2$ is an $n$-framed longitude, and $\beta_2$ is a meridian. Let $(Y_1, \alpha_1, \beta_1) \cup (Y_2, \alpha_2, \beta_2)$ denote the result of gluing $Y_1$ to $Y_2$ by identifying $\alpha_1$ with $\alpha_2$ and $\beta_1$ with $\beta_2$. Note that $(Y_1, \alpha_1, \beta_1) \cup (Y_2, \alpha_2, \beta_2)$ is homeomorphic to $(p_1, \dots, p_r)$-surgery along $K_{r, rn}$. We identify the slope $p \alpha_i + q \beta_i$ on $\partial Y_i$ with the (extended) rational number $\frac{p}{q} \in \Q \cup \{ \frac{1}{0}\}$.

The following lemma gives a description of $\CFD(Y_1, \alpha_1, \beta_1)$ in terms of the standard notation defined in \cite[Section 3.2]{HanselmanWatson}.

\begin{lemma}\label{lem:CFDY1}
The invariant $\CFD(Y_1, \alpha_1, \beta_1)$ can be written in standard notation as a product of $d_{k_i}$ where 
\begin{enumerate}
\item $k_i \leq 0$ for all $i$,
\item $k_i=0$ for at least one $i$,
\item $k_i=-r$ for exactly one $i$.
\end{enumerate}
\end{lemma}

\begin{proof}
The computation is similar to the example in \cite[Section 6.5]{HanselmanWatson}. A plumbing tree $\Gamma$ for $Y_1$ is given in Figure \ref{fig:plumbingtree}.  We first consider the plumbing tree $\Gamma_i$ in Figure \ref{fig:Gammai}. We will build $\Gamma$ by merging the $\Gamma_i$, $i=1, \dots, r$.

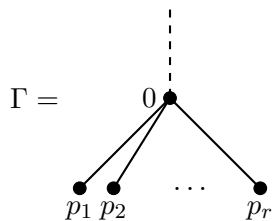
\begin{figure}[ht]
\begin{tikzpicture}[scale=0.6]

\tikzstyle{every node}=[inner sep=0,outer sep=0]
	
	\filldraw (-2, 0) circle (4pt) node[] (a){};
	\filldraw (-1.25, 0) circle (4pt) node[] (b){};
	\filldraw (2, 0) circle (4pt) node[] (c){};
	\filldraw (0, 2) circle (4pt) node[] (r){};
	\filldraw (0, 4) node[] (x){};

	\node [below, yshift=-5pt] at (a) {$p_1$};
	\node [below, yshift=-5pt] at (b) {$p_2$};
	\node [below, yshift=-5pt] at (c) {$p_r$};
	\node [left, xshift=-5pt] at (r) {$0$};
	\node at (0.5,0) {$\dots$};
	\node at (-3, 2) {$\Gamma=$};
		
	\draw [thick, -] (r) -- (a);
	\draw [thick, -] (r) -- (b);
	\draw [thick, -] (r) -- (c);
	\draw [thick, dashed, -] (r) -- (x);
\end{tikzpicture}
\caption{The plumbing tree $\Gamma$.}
\label{fig:plumbingtree}
\end{figure}

We proceed as in \cite[Section 6.5]{HanselmanWatson}. Start with a loop $(d_0)$ representing the tree $\Gamma_0$ in Figure \ref{fig:Gamma0}. We have that $\Gamma_i = \cE(\cT^{p_i}(\Gamma_0))$ so by \cite[Sections 3.3 and 6.3]{HanselmanWatson}:
\begin{align*} 
	\CFD(\Gamma_i) &= \sce(\sct^{p_i}((d_0))) \\
		&= \sce((d_{p_i})) \\
		&= (d^*_{-p_i}) \\
		&\sim (d_{-1} \underbrace{d_0 \dots d_0}_{p_i}).
\end{align*}

\begin{figure}[ht]
\subfigure[]{
\begin{tikzpicture}[scale=0.6]
\tikzstyle{every node}=[inner sep=0,outer sep=0]
	\filldraw (-2, 0) circle (4pt) node[] (a){};
	\filldraw (0, 0) circle (4pt) node[] (r){};
	\filldraw (2, 0) node[] (x){};
	\node [below, yshift=-5pt] at (a) {$p_i$};
	\node [below, yshift=-5pt] at (r) {$0$};
	\node at (-3.5, 0) {$\Gamma_i=$};
	\draw [thick, -] (r) -- (a);
	\draw [thick, dashed, -] (r) -- (x);
	\node at (0, -2) {};
\end{tikzpicture}
\label{fig:Gammai}
}
\hspace{50pt}
\subfigure[]{
\begin{tikzpicture}[scale=0.6]
\tikzstyle{every node}=[inner sep=0,outer sep=0]
	\filldraw (-9, 0) circle (4pt) node[] (s){};
	\filldraw (-7, 0) node[] (y){};
	\node [below, yshift=-5pt] at (s) {$0$};
	\node at (-10.5, 0) {$\Gamma_0=$};
	\draw [thick, dashed, -] (s) -- (y);
	\node at (-8, -2) {};
\end{tikzpicture}
\label{fig:Gamma0}
}
\caption{Left, the plumbing tree $\Gamma_i$. Right, the plumbing tree $\Gamma_0$.}
\label{fig:miniplumbingtree}
\end{figure}
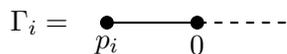
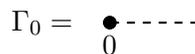

We then have that $\Gamma = \cM(\Gamma_2, \cM(\Gamma_2, \dots, \cM(\Gamma_{p_{r-1}}, \Gamma_{p_r})))$.
By \cite[Proposition 6.4]{HanselmanWatson}, we have that $\CFD(\Gamma)$ is a represented by a product of $d_{k_i}$ where $k_i \leq 0$ for all $i$ and $k_i=0$ for at least one $i$ since each $p_i > 0$. Moreover, $d_{-r}$ appears exactly once in the product, since we performed $r-1$ merges. This completes the proof of the lemma.
\end{proof}

\begin{lemma}\label{lem:Y1}
The slope $1$ is not a strict L-space slope on $(Y_1, \alpha_1, \beta_1)$.
\end{lemma}

\begin{proof}
We will apply a positive Dehn twist to $(Y_1, \alpha_1, \beta_1)$ to obtain $(Y_1, \alpha_1, \beta_1 + \alpha_1)$. We will show that $0$ is not a strict L-space slope on $(Y_1, \alpha_1, \beta_1 + \alpha_1)$, and hence $1$ is not a strict L-space slope on $(Y_1, \alpha_1, \beta_1)$.

By \cite[Proposition 6.1]{HanselmanWatson}, we have that $\CFD(Y_1, \alpha_1, \beta_1 + \alpha_1)$ can be obtained by applying $\sct$ to a loop representative of $\CFD(Y_1, \alpha_1, \beta_1)$. Since $\sct(d_k) = d_{k+1}$, it follows from Lemma \ref{lem:CFDY1} that $\CFD(Y_1, \alpha_1, \beta_1 + \alpha_1)$ can be written in standard notation as a product of $d_{k_i}$ with $k_i \leq 1$ for all $i$, $k_i=1$ for at least one $i$, and $k_i=1-r$ for exactly one $i$.

We claim that if a loop $\bell$ contains both positive and negative $d_k$ segments (i.e., both $d_i, i>0$ and $d_j, j<0$), then in dual notation $\bell$ contains at least one $a^*_i$ or $b^*_j$ segment. Indeed, suppose by contradiction that $\bell$ has no $a^*_i$ or $b^*_j$. Then $\bell$ consists of only $d^*_i$ segments, $i \in \Z$. It is straightforward to see (for example, by considering the segments as drawn in \cite[Figure 1]{HanselmanWatson}) that one cannot obtain a loop containing both positive and negative $d_k$ segments from $d^*_i$ segments, $i \in \Z$. This completes the proof of the claim.

Furthermore, note that $\CFD(Y_1, \alpha_1, \beta_1 + \alpha_1)$ consists of simple loops (see Definition 4.19 of \cite{HanselmanWatson}). Then by \cite[Proposition 4.24]{HanselmanWatson}, in dual notation $\bell$ has no $a^*_k$ or $b^*_k$ segments for $k<0$. It now follows from Proposition 4.18 of \cite{HanselmanWatson} that $0$ is not a strict L-space slope for $\CFD(Y_1, \alpha_1, \beta_1 + \alpha_1)$. Therefore, $1$ is not a strict L-space slope on $(Y_1, \alpha_1, \beta_1)$, as desired.
\end{proof}

\begin{remark}\label{rk:slopesY1}
Note that by Proposition 4.18 of \cite{HanselmanWatson}, we have that $0$ and $\infty$ are strict L-space slopes on  $(Y_1, \alpha_1, \beta_1)$. Since $1$ is not a strict L-space slope, it follows from Corollary 4.5 of  \cite{HanselmanWatson} that the interval of L-space slopes of $(Y_1, \alpha_1, \beta_1)$ contains the interval $[-\infty, 0]$.
\end{remark}

\begin{remark}
An alternative proof of Lemma \ref{lem:Y1} follows from \cite[Theorem 1.1]{LiscaStipsicz}. Indeed, by setting $r_i=1/p_i$ and $e_0=-1$ in Figure 1 of \cite{LiscaStipsicz}, we see that $M(-1; 1/p_1, \dots, 1/p_r)$ is not an L-space, hence neither is $M(1; -1/p_1, \dots, -1/p_r)$, which is homeomorphic to filling $(Y_1, \alpha_1, \beta_1)$ along a curve of slope 1.
\end{remark}

\begin{lemma}\label{lem:Y2}
Let $K$ be an L-space knot. If $n < 2g(K)-1$, then $1$ is not a strict L-space slope on the $n$-framed knot complement $(Y_2, \alpha_2, \beta_2)$.
\end{lemma}

\begin{proof}
Since $K$ is an L-space knot, we have that $S^3_K(p/q)$ is an L-space exactly when $p/q \geq 2g(K)-1$. Since $\alpha_2$ is an $n$-framed longitude, it follows that the interval of strict L-space slopes on $(Y_2, \alpha_2, \beta_2)$ is $(0, \frac{1}{2g(K)-1-n})$, that is, the reciprocal of the interval $(2g(K)-1-n, \infty)$.
\end{proof}

\begin{proof}[Proof of Proposition \ref{prop:nonLspace}]
The result now follows from \cite[Theorem 1.3]{HanselmanWatson} combined with Lemmas \ref{lem:Y1} and \ref{lem:Y2}; the slope $1$ is not a strict L-space slope on either $(Y_1, \alpha_1, \beta_1)$ or $(Y_2, \alpha_2, \beta_2)$, and so the resulting manifold $(Y_1, \alpha_1, \beta_1) \cup (Y_2, \alpha_2, \beta_2)$, which is $(p_1, \dots, p_r)$-surgery on $K_{r, rn}$, is not an L-space.
\end{proof}

\begin{remark}
One can use similar methods to provide an alternate proof that $K_{r,rn}$ is an L-space link if $K$ is an L-space knot and $n \geq 2g(K)-1$.
Indeed, if $K$ is an L-space knot, then the interval of strict L-space slopes on the $n$-framed knot complement $(Y_2, \alpha_2, \beta_2)$ is $(0, \frac{1}{2g(K)-1-n})$ if $n \leq 2g(K)-1$ and $(0, \infty] \cup [-\infty, \frac{1}{2g(K)-1-n})$ if $n > 2g(K)-1$. Hence if $n \geq 2g(K)-1$, then the interval of strict L-space slopes on $(Y_2, \alpha_2, \beta_2)$ contains the interval $(0, \infty)$. By Remark \ref{rk:slopesY1}, we have that the interval of strict L-space slopes on $(Y_1, \alpha_1, \beta_1)$ contains $[-\infty, 0]$. Therefore, by \cite[Theorem 1.4]{HanselmanWatson}, if $n \geq 2g(K)=1$, then the result of positive surgery (i.e., each surgery coefficient is positive) on $K_{r, rn}$ is an L-space.
\end{remark}

\section{A spectral sequence for L-space links}

In this section we review some material from \cite{gn}. Given $u,v\in \Z^r$, we write $u\preceq v$ if $u_i\leq v_i$ for all $i$,
and $u\prec v$ if $u\preceq v$ and $u\neq v$. Recall that we work with $\F=\Z/2\Z$ coefficients.

\begin{definition}
Given a $r$-component oriented link $L$, we define an affine lattice over $\Z^r$:
$$
\HH(L) =\bigoplus_{i=1}^{r}\HH_i(L),\qquad  \HH_i(L)=\Z+\frac{1}{2}\lk(L_i,L-L_i).
$$
\end{definition}

Let us recall that the Heegaard-Floer complex for a $r$-component link $L$ is naturally filtered by the subcomplexes $A^{-}_{L}(L;v)$ of $\F[U_1,\ldots,U_r]$-modules for $v\in \HH(L)$.
Such a subcomplex is spanned by the generators in the Heegaard-Floer complex of Alexander filtration less than or equal to $v$ in the natural partial order on $\HH(L)$.
The group $\HFL(L,v)$ can be defined as the homology of the associated graded complex:
\begin{equation}
\label{def of HFL}
\HFL(L,v)=H_{*}\left(A^{-}(L;v)/\sum_{u\prec v}A^{-}(L;u)\right).
\end{equation}

One can forget a component $L_r$ in $L$ and consider the $(r-1)$-component link $L-L_r$. 
There is a natural forgetful map $\pi_r:\HH(L)\to \HH(L-L_r)$ defined by the equation:
$$
\pi_r(v_1,\ldots,v_r)=\left(v_1-\lk(L_1,L_r)/2,\ldots,v_{r-1}-\lk(L_{r-1},L_r)/2\right).
$$
Similarly, one can define a map $\pi_{L'}:\HH(L)\to \HH(L')$ for every sublink $L'\subset L$.
Furthermore, for large $v_r\gg 0$ the subcomplexes $A^{-}(L;v)$ stabilize, and by \cite[Proposition 7.1]{OSlinks} one has
a natural homotopy equivalence $A^{-}(L;v)\sim  A^{-}(L-L_r;\pi_r(v))$. More generally, for a sublink $L'=L_{i_1}\cup\ldots \cup L_{i_{r'}}$ one gets
\begin{equation}
\label{projection for a-}
A^{-}(L';\pi_{L'}(v))\sim  A^{-}(L;v),\ \text{if}\ v_i\gg 0\ \text{for}\ i\notin\{i_1\ldots,i_{r'}\}.
\end{equation}

We will use the ``inversion theorem'' of \cite{gn}, expressing the $h$-function of a link in terms of the  Alexander polynomials of its sublinks, or, equivalently, the Euler characteristics of their Heegaard-Floer homology. Define $\chi_{L,v}:=\chi(\HFL(L,v))$. Then by \cite{OSlinks}
$$\chi_L(t_1,\ldots,t_r):=\sum_{v\in \HH(L)}\chi_{L,v}t_1^{v_1}\cdots t_r^{v_r}=\begin{cases}
(t_1\cdots t_r)^{1/2}\Delta(t_1,\ldots,t_r), & \text{if}\ r>1\\
\Delta(t)/(1-t^{-1}), & \text{if}\ r=1,
\end{cases}
$$
where $\Delta(t_1,\ldots,t_r)$ denotes the {\em symmetrized} Alexander polynomial. 

\begin{remark}
We choose the factor $(t_1\cdots t_r)^{1/2}$ to match more established conventions on the gradings for the hat-version of link Floer homology.
For example, the Alexander polynomial of the Hopf link equals $1$, and one can check \cite{OSlinks} that $\HFLhat$ is supported in Alexander degrees 
$(\pm\frac{1}{2},\pm\frac{1}{2})$. Since the maximal Alexander degrees in  $\HFLhat$ and $\HFL$ coincide, one gets 
$\chi_{T(2,2)}(t_1,t_2)=t_1^{1/2}t_2^{1/2}$.
\end{remark}

The following ``large surgery theorem'' underlines the importance of $A^{-}(L;v)$.

\begin{theorem}[\cite{MO}]
The homology of  $A^{-}(L;v)$ is isomorphic to the Heegaard-Floer homology of a large surgery on $L$ with $spin_c$-structure specified by $v$. In particular, if $L$ is an L-space link, then $H_{*}(A^{-}(L,v))\simeq \F[U]$ for all $v$ and all $U_i$ are homotopic to each other on the subcomplex  $A^{-}(L;v)$.
\end{theorem}

One can show that for L-space links the inclusion $h_{v}:A^{-}(L,v)\hookrightarrow A^{-}(S^3)$ is injective on homology,
so it is multiplication by $U^{h_L(v)}$. Therefore the generator of $H_{*}(A^{-}(L,v))\simeq \F[U]$ has homological degree $-2h_L(v)$.
The function $h_L(v)$ will be called the {\em $h$--function} for an L--space link $L$. In \cite{gn} it was called an ``HFL-weight function''.

Furthermore, if $L$ is an L-space link, then for large $N\in \HH(L)$ one has
$$
\chi\left(A^{-}(L;N)/A^{-}(L,v)\right)=h_L(v).
$$ 
Hence, by \eqref{def of HFL} and the inclusion-exclusion formula one can write:
\begin{equation}
\label{chi from h}
\chi_{L,v}=\sum_{B\subset \{1,\ldots,r\}}(-1)^{|B|-1}h_L(v-e_B),
\end{equation}
where $e_B$ denotes the characteristic vector of the subset $B\subset \{1,\ldots,r\}$. 
Furthermore, by \eqref{projection for a-} for a sublink $L'=L_{i_1}\cup\ldots \cup L_{i_{r'}}$ one gets
\begin{equation}
\label{projection for h}
h_{L'}(\pi_{L'}(v))=h_L(v),\ \text{if}\ v_i\gg 0\ \text{for}\ i\notin\{i_1\ldots,i_{r'}\}.
\end{equation}
For $r=1$ equation \eqref{chi from h} has the form $\chi_{L,v}=h(v-1)-h(v)$, so $h(v)$ can be easily reconstructed from the Alexander polynomial: $h_L(v)=\sum_{u\ge v+1}\chi_{L,v}.$ For $r>1$, one can also show that equation \eqref{chi from h} (together with the boundary conditions \eqref{projection for h}) has a unique solution, which is given by the following theorem: 

\begin{theorem}[\cite{gn}]
\label{th:invertion}
The $h$-function of an L-space link is determined by the Alexander polynomials of its sublinks as following:
\begin{equation}
\label{invertion}
h_{L}(v_1,\ldots,v_r)=\sum_{L'\subseteq L}(-1)^{r'-1}\sum_{u\succeq \pi_{L'}(v+\mathbf{1})}\chi_{L',u},
\end{equation}
where the sublink $L'$ has $r'$ components and $\mathbf{1}=(1,\ldots,1)$.
\end{theorem}

Given an L-space link, we construct a spectral sequence whose $E_2$ page can be computed from the multi-variable 
Alexander polynomial by an explicit combinatorial procedure, and whose $E_{\infty}$ page coincides with the group $\HFL$.
The complex \eqref{def of HFL} is quasi-isomorphic to the iterated cone: 
$$
\mathcal{K}(v)=\bigoplus_{B\subset \{1,\ldots,r\}}A^{-}(L,v-e_B),
$$
where the differential consists of two parts: the first acts in each summand and the second acts by inclusion maps between summands.
There is a spectral sequence naturally associated to this construction. Its $E_1$ term equals 
$$
E_1(v)=\bigoplus_{B\subset \{1\ldots,r\}}H_{*}(A^{-}(L,v-e_B))=\bigoplus_{B\subset \{1\ldots,r\}}\F[U]\langle z(v-e_B)\rangle,
$$
where $z(u)$ is the generator of $H_{*}(A^{-}(L,u))$ of degree $-2h_L(u)$.
The next differential $\partial_1$ is induced by inclusions and reads as:
\begin{equation}
\label{def d1}
\partial_1(z(v-e_B))=\sum_{i\in B} U^{h(v-e_B)-h(v-e_{B-i})}z(v-e_B+e_i).
\end{equation}
We obtain the following result.

\begin{theorem}[\cite{gn}]
\label{spectral}
Let $L$ be an L-space link with $r$ components and let $h_L(v)$ be the corresponding $h$-function.
Then there is a spectral sequence with $E_2(v)=H_{*}(E_1,\partial_1)$ and $E_{\infty}\simeq \HFL(L,v)$.
\end{theorem}

\begin{remark}
\label{rem:degrees}
Let us write more precisely the bigrading on the $E_2$ page. The $E_1$ page is naturally bigraded as follows: a generator
$U^{m}z(v-e_B)$ has {\em cube degree} $|B|$ and its homological degree {\em in $A^{-}(L,v-e_B)$} equals $-2m-2h(v-e_B)$.
In short, we will write
$$
\bideg\left(U^{m}z(v-e_B)\right)=(|B|,-2m-2h(v-e_B)).
$$
The homological degree of the same generator in $E_1(v)$ equals the sum of these two degrees. The differential $\partial_1$
has bidegree $(-1,0)$, and, more generally, the differential $\partial_k$ in the spectral sequence has bidegree $(-k,k-1)$. 
\end{remark}

In the next section we will compute the $E_2$ page for cable L-space links and show that $E_2=E_{\infty}$.
Let us discuss the action of the operators $U_i$ on the $E_2$ page. Recall that $U_i$ maps $A^{-}(L,v)$ to $A^{-}(L,v-e_i)$,
and in homology one has:
\begin{equation}
\label{U on generators}
U_iz(v)=U^{1-h(v-e_i)+h(v)}z(v-e_i).
\end{equation}
Since $U_i$ commutes with the inclusions of various $A^{-}$, we get the following result.

\begin{proposition}
\label{U on HFL}
Equation \eqref{U on generators} defines a chain map from $\mathcal{K}(v)$ to $\mathcal{K}(v-e_i)$ commuting with the differential $\partial_1$,
so we have a well-defined combinatorial map 
$$U_i:H_{*}(E_1(v),\partial_1)\to H_{*}(E_1(v-e_i),\partial_1).$$ 
If $E_2=E_{\infty}$ then
one obtains $U_i:\HFL(L,v)\to \HFL(L,v-e_i)$.
\end{proposition}

Furthermore, by the definition of $\HFLhat$ \cite[Section 4]{OSlinks} one gets:
$$
\HFLhat(L,v)=H_{*}\left(A^{-}(L,v)/\left[\sum_{i=1}^{r}A^{-}(v-e_i)\oplus \sum_{i=1}^{r}U_iA^{-}(v+e_i)\right]\right).
$$ 
This implies the following result:

\begin{proposition}
\label{spectral for hat}
There is a spectral sequence with $E_1$ page
$$
\widehat{E}_1=\bigoplus_{B\subset \{1,\ldots,r\}}\HFL(L,v+e_B)
$$ 
and converging to $\widehat{E}_{\infty}=\HFLhat(L,v)$. The differential $\widehat{\partial}_1$ is given 
by the action of $U_i$ induced by \eqref{U on generators}.
\end{proposition}

\section{Heegaard-Floer homology for cable links}

\subsection{The Alexander polynomial and \texorpdfstring{$h$}{h}--function}

The Alexander polynomial of cable knots and links is given by the following well-known formula:
\begin{equation}
\label{alexander cabling}
\Delta_{K_{rm,rn}}(t_1,\ldots,t_r)=\Delta_K(t_1^{m}\cdots t_r^{m})\cdot \Delta_{T(rm,rn)}(t_1,\ldots,t_r),
\end{equation}
where $T(rm,rn)$ denotes the $(rm,rn)$ torus link. Throughout, let $\t=t_1\cdots t_r$ and $l=mn$.
\begin{lemma}
The generating functions for the Euler characteristics of $\HFL$ for $K_{rm,rn}$ and $K_{m,n}$ are related by the following equation:
\begin{equation}
\label{cable links vs knots}
\chi_{K_{rm,rn}}(t_1,\ldots,t_r)=\chi_{K_{m,n}}(\t)\cdot (\t^{l/2}-\t^{-l/2})^{r-1}.
\end{equation}
\end{lemma}
\begin{proof}
The statement follows from the identity \eqref{alexander cabling} and the expression for the Alexander polynomials of torus links:
$$
\chi_{T(rm,rn)}(t_1,\ldots,t_r)=\frac{(\t^{mn/2}-\t^{-mn/2})^{r}}{(\t^{m/2}-\t^{-m/2})(\t^{n/2}-\t^{-n/2})}.
$$
\end{proof}

\begin{remark}
The Alexander polynomial is determined up to a sign. By \eqref{cable links vs knots}, the multivariable Alexander polynomial of a cable link
is supported on the diagonal, so one can fix the sign by requiring its top coefficient to be positive.
\end{remark}

From now on we will assume that $K$ is an L-space knot and $n/m\ge 2g(K)-1$, so $K_{rm,rn}$ is an L-space link for all $r$. To simplify notation, we define 
$h_{rm,rn}(v)=h_{K_{rm,rn}}(v)$ and  $\chi_{rm,rn}(v)=\chi_{K_{rm,rn},v}.$
Let $c=l(r-1)/2$.

\begin{theorem}
Suppose that $v_1\le v_2\le \ldots \le v_r$. Then the following equation holds:
\begin{equation}
\label{h for link}
h_{rm,rn}(v_1,\ldots,v_r)=h_{m,n}(v_1-c)+h_{m,n}(v_2-c+l)+\ldots+h_{m,n}(v_r-c+(r-1)l).
\end{equation}
\end{theorem}

\begin{proof}
We will use Theorem \ref{th:invertion} to compute $h(v)$. Let $L'$ be a sublink of $K_{rm,rn}$ with $r'$ components, i.e., $L'=K_{r'm,r'n}$.
By \eqref{cable links vs knots}, one has
$$
\chi_{K_{r'm,r'n}}(t_1,\ldots,t_{r'})=\chi_{K_{m,n}}(\t)\cdot \t^{l(r'-1)/2}\sum_{j=0}^{r'-1}(-1)^{j}\binom{r'-1}{j}\t^{-lj},
$$
hence $\chi_{L',u}$ does not vanish only if $u=(s,\ldots,s)$, and 
$$
\chi_{L',s,\ldots,s}=\sum_{j=0}^{r'-1}(-1)^{j}\binom{r'-1}{j}\chi_{m,n}(s-l(r'-1)/2+lj).
$$
Therefore
\begin{align*}
\sum_{u\succeq \pi_{L'}(v+\mathbf{1})}\chi_{L',u} &=\sum_{s>\max(\pi_{L'}(v))}\sum_{j=0}^{r'-1}(-1)^{j}\binom{r'-1}{j}\chi_{m,n}(s-l(r'-1)/2+lj) \\
	&=\sum_{j=0}^{r'-1}(-1)^{j}\binom{r'-1}{j}h_{m,n}(\max(\pi_{L'}(v))-l(r'-1)/2+lj).
\end{align*}
Furthermore, if $L'=L_{i_1}\cup\ldots\cup L_{i_{r'}}$ then $\pi_{L'}(v)=(v_{i_1}-l(r-r')/2,\ldots,v_{i_{r'}}-l(r-r')/2)$, so
$$
\max(\pi_{L'}(v))=\max(v_{i_1},\ldots,v_{i_r'})-l(r-r')/2=\max(v_{L'})-l(r-r')/2.
$$
This means that \eqref{invertion} can be rewritten as follows:
\begin{align*}
h_{rm,rn}(v_1,\ldots,v_r)&=\sum_{L', j}(-1)^{r'-1+j}\binom{r'-1}{j}h_{m,n}(\max(v_{L'})-l(r-1)/2+lj) \\
&=\sum_{i,j}h_{m,n}(v_i-l(r-1)/2+lj)\sum_{L':v_i=\max(v_{L'})}(-1)^{r'-1+j}\binom{r'-1}{j}.
\end{align*}
One can check that the inner sum vanishes unless $j=i-1$ (recall that $v_1\le v_2\le \ldots \le v_r$), so one gets
$$
h_{rm,rn}(v_1,\ldots,v_r)=\sum_{i}h_{m,n}(v_i-l(r-1)/2+l(i-1)).
$$
\end{proof}





\begin{lemma}
\label{symmetry for h}
The following identity holds:
$$
h_{rm,rn}(-v_1,\ldots,-v_r)=h_{rm,rn}(v_1,\ldots,v_r)+(v_1+\ldots+v_r).
$$
\end{lemma}

\begin{proof}
Suppose that $v_1\le v_2\le\ldots\le v_r$. Then $-v_1\ge -v_2\ge\ldots\ge -v_r$.  Therefore
\begin{align*}
h_{rm,rn}(-v_1,\ldots,-v_r) &=\sum_{i=1}^{r}h_{m,n}(-v_i-l(r-1)/2+l(r-i))\\
	&=\sum_{i=1}^{r}h_{m,n}(-v_i+l(r-1)/2-l(i-1)).
\end{align*}
It is known (e.g., \cite{hlz}) that for all $x$, 
$$
h_{m,n}(-x)=h_{m,n}(x)+x,
$$
hence
$$
h_{m,n}(-v_i+l(r-1)/2-l(i-1))=h_{m,n}(v_i-l(r-1)/2+l(i-1))+(v_i-l(r-1)/2+l(i-1)).
$$
Finally, $\sum_{i=1}^{r}(-l(r-1)/2+l(i-1))=0$.
\end{proof} 

\begin{lemma}
\label{h on diagonal}
One has $h_{rm,rn}(k,k\ldots,k)=\hh(k)$, where $\hh(k)$ is defined by \eqref{def hh}.
\end{lemma}

\begin{proof}
Indeed, by \eqref{h for link} we have 
$$
h_{rm,rn}(k,\ldots,k)=h_{m,n}(k-l(r-1)/2)+h_{m,n}(k-l(r-1)/2+l)+\ldots+h_{m,n}(k+l(r-1)/2),
$$
so
$$
\sum_{k}h_{rm,rn}(k,\ldots,k)t^k=(t^{-l(r-1)/2}+\ldots+t^{l(r-1)/2})\sum_{k}h_{m,n}(k)t^k=\frac{(t^{lr/2}-t^{-lr/2})}{(t^{l/2}-t^{-l/2})}\cdot \frac{t^{-1}\Delta_{m,n}(t)}{(1-t^{-1})^2}.
$$
\end{proof}
 
For the rest of this section we will assume that $n/m>2g(K)-1$. 

\begin{lemma}
\label{HFL after l}
If $v\le g(K_{m,n})-l$, then $\HFK(K_{m,n},v)\simeq \F$. 
\end{lemma}
\begin{proof}
By \cite[Theorem 1.10]{HeddencablingII}, $K_{m,n}$ is an L-space knot and hence by \cite{OSlens}
$$g(K_{m,n})=\tau(K_{m,n}), \qquad g(K)=\tau(K).$$
By \cite{Shibuya}, we have:
$$g(K_{m,n})=mg(K)+\frac{(m-1)(n-1)}{2},$$
so for $n/m>2g(K)-1$ we have 
$$2g(K_{m,n})=2mg(K)+mn-m-n+1<mn+1,$$
hence $l=mn\ge 2g(K_{m,n})$. 
On the other hand, it is well-known that for $v\le -g(K_{m,n})$ one has $\HFK(K_{m,n},v))\simeq \F$. 
\end{proof} 

We will use the function $\beta$ defined by \eqref{def hh}.

\begin{lemma}
If $\beta(k)=-1$ then $\HFK(K_{m,n},k-c)=0$. Otherwise
\begin{equation}
\label{beta as max}
\beta(k)=\max\{j:0\le j\le r-1,\ \HFK(K_{m,n},k-c+lj)\simeq \F\}.
\end{equation}
\end{lemma}

\begin{proof}
By \eqref{def hh} and Lemma \ref{h on diagonal} we have
$$
\beta(k)+1=h_{rm,rn}(k-1,\ldots,k-1)-h_{rm,rn}(k,\ldots,k)=\sum_{j=0}^{r-1}\left(h_{m,n}(k-1-c+lj)-h_{m,n}(k-c+lj)\right).
$$
Note that 
$h_{m,n}(k-1-c+lj)-h_{m,n}(k-c+lj)=\dim \HFK(K_{m,n},k-c+lj)\in \{0,1\}.$
If $\HFK(K_{m,n},k-c+lj)\simeq \F$ then $k-c+lj\le g(K_{m,n})$,  so by Lemma \ref{HFL after l}  $\HFK(K_{m,n},k-c+lj')\simeq \F$ for all $j'<j$.
Therefore, if $\HFK(K_{m,n},k-c)=0$ then $\beta(k)=-1$, otherwise 
$$
\HFK(K_{m,n},k-c+lj)=\begin{cases}
\F&\text{if}\  j\le \beta(k),\\
0&\text{if}\  j> \beta(k).\\
\end{cases}
$$
\end{proof}

Suppose that $v_1=\ldots=v_{\lambda_1}=u_1,v_{\lambda_1+1}=\ldots=v_{\lambda_1+\lambda_2}=u_2,\ldots,v_{\lambda_1+\ldots+\lambda_{s-1}+1}=\ldots=v_r=u_s$ where $u_1<u_2<\ldots <u_s$ and $\lambda_1+\ldots+\lambda_s=r$. We will abbreviate this as $v=(u_1^{\lambda_1},\ldots,u_s^{\lambda_s}).$

\begin{lemma}
\label{lem:zero}
Suppose that $\beta(u_s)<r-\lambda_s$. 
Then for any subset $B \subset \{1, \dots, r-1\}$ one has
$h_{rm,rn}(v-e_B)=h_{rm,rn}(v-e_B-e_r)$.
\end{lemma}
\begin{proof}
To apply \eqref{h for link}, one needs to reorder the components of the vectors $v-e_B$ and $v-e_B-e_r$.
Note that in both cases the last (largest) $\lambda_s$ components are equal either to $u_s$ or to $u_s-1$, and
the corresponding contributions to $h_{rm,rn}$ are equal to $h_{m,n}(u_s-c+l(r-\lambda_s)+lj)$ or to $h_{m,n}(u_s-c+l(r-\lambda_s)+lj-1)$,
respectively ($j=0,\ldots,\lambda_s-1$). On the other hand, by \eqref{beta as max}
one has
\[
\HFK(K_{m,n},u_s-c+l(r-\lambda_s)+lj)=0
\]
and so
\[
 h_{m,n}(u_s-c+l(r-\lambda_s)+lj-1)=h_{m,n}(u_s-c+l(r-\lambda_s)+lj).
\]
\end{proof}

\begin{lemma}
\label{lem:h for nonzero}
If $\beta(u_s)\ge r-\lambda_s$ then 
$h_{rm,rn}(v)=\hh(u_s)+ru_s-|v|.$
\end{lemma} 

\begin{proof}
Since $\beta(u_s)\ge r-\lambda_s$, we have $\HFK(K_{m,n},u_s-c+l(r-\lambda_s))\simeq \F$, so $$u_s-c+l(r-\lambda_s)\le g(K_{m,n}).$$
For $i\le r-\lambda_s$ we get 
$$v_i-c+l(i-1)<u_s-c+l(i-1)\le u_s-c+l(r-\lambda_s)-l\le g(K_{m,n})-l,$$
so by Lemma \ref{HFL after l}, $\HFK(K_{m,n},w)\simeq \F$ for all $w\in [v_i-c+l(i-1),u_s-c+l(i-1)]$, and
$$h_{m,n}(v_i-c+l(i-1))=h_{m,n}(u_s-c+l(i-1))+(u_s-v_i).$$
Now the statement follows from Lemma \ref{h for link}.
\end{proof}

\begin{lemma}
\label{lem:nonzero}
Suppose that $\beta(u_s)\ge r-\lambda_s$. 
Then for any subsets $B'\subset \{1,\ldots,r-\lambda_s\}$ and $B''\subset \{r-\lambda_s+1,\ldots,r\}$ one has
$$h_{rm,rn}(v-e_{B'}-e_{B''})=h_{rm,rn}(v)+|B'|+\min(|B''|,\beta(u_s)-r+\lambda_s+1).$$
\end{lemma}

\begin{proof}
Since $\HFK(K_{m,n},u_s-c+l(r-\lambda_s))\simeq \F$, we have $u_s-c+l(r-\lambda_s)\le g(K_{m,n})$, so for all $i\le r-\lambda_s$
one has $v_i-c+l(i-1)<u_s-c+l(r-\lambda_s)-l\le g(K_{m,n})-l$, and by Lemma \ref{HFL after l} $\HFK(K_{m,n},v_i-c+l(i-1))\simeq \F$,
and $h_{m,n}(v_i-1-c+l(i-1))=h_{m,n}(v_i-c+l(i-1))+1.$ Therefore $h_{rm,rn}(v-e_{B'}-e_{B''})=|B'|+h_{rm,rn}(v-e_{B''}).$ Finally,
\begin{align*}
h_{rm,rn}(v-e_{B''})-h_{rm,rn}(v)&=\sum_{j=0}^{|B''|}\left(h_{m,n}(u_s-1-c+l(r-\lambda_s)+lj)-h_{m,n}(u_s-c+l(r-\lambda_s)+lj\right) \\
	&=\min(|B''|,\beta(u_s)-r+\lambda_s+1).
\end{align*}
\end{proof}

\subsection{Spectral sequence for \texorpdfstring{$\HFL$}{HFL-}}
\label{sec:minus}

\begin{definition}
Let $\cE_r$ denote the exterior algebra over $\F$ with variables $z_1,\ldots,z_r$. Let us define the {\em cube differential} on $\cE_r$ by the equation
$$
\partial(z_{\alpha_1}\wedge\ldots\wedge z_{\alpha_k})=\sum_{j=1}^{k}z_{\alpha_1}\wedge\ldots \wedge \widehat{z_{\alpha_j}}\wedge \ldots \wedge z_{\alpha_k},
$$
and the {\em $b$-truncated differential} on $\cE_r[U]$ by the equation
$$
\partial^{(b)}(z_{\alpha_1}\wedge\ldots\wedge z_{\alpha_k})=\begin{cases}
U \partial(z_{\alpha_1}\wedge\ldots\wedge z_{\alpha_k}), & \text{if}\ k\le b\\
 \partial(z_{\alpha_1}\wedge\ldots\wedge z_{\alpha_k}), & \text{if}\ k> b.\\
\end{cases}
$$
\end{definition}

More invariantly, one can define the {\em weight} of a monomial $z_{\alpha}=z_{\alpha_1}\wedge\ldots\wedge z_{\alpha_k}$ as
$w(z_{\alpha})=\min(|\alpha|,b)$, and the $b$-truncated differential  is given by the equation:
\begin{equation}
\label{d b}
\partial^{(b)}(z_{\alpha})=\sum_{i\in \alpha} U^{w(\alpha)-w(\alpha-\alpha_i)}z_{\alpha-\alpha_i}.
\end{equation}
Indeed, $w(\alpha)-w(\alpha-\alpha_i)=1$ for $|\alpha|\le b$ and $w(\alpha)-w(\alpha-\alpha_i)=0$ for $|\alpha|>b$. 

\begin{definition}
Let $\cEr_r\subset \cE_r$ be the subalgebra of $\cE_r$ generated by the differences $z_i-z_j$ for all $i\neq j$.
\end{definition}

\begin{lemma}
\label{Ker d}
The kernel of the cube differential $\partial$ on $\cE_r$ coincides with $\cEr_r$.
\end{lemma}

\begin{proof}
It is clear that $\partial(z_i-z_j)=0$, and Leibniz rule implies vanishing of $\partial$ on $\cEr_r$. 
Let us prove that $\Ker\partial\subset \cEr_r$. Since $(\cE_r,\partial)$ is acyclic, it is sufficient to prove that
the image of every monomial $z_{\alpha_1}\wedge\cdots\wedge z_{\alpha_k}$ is contained in $\cE_r$. Indeed, one can check that
$$
\partial(z_{\alpha_1}\wedge\cdots\wedge z_{\alpha_k})=(z_{\alpha_2}-z_{\alpha_1})\wedge\cdots\wedge (z_{\alpha_k}-z_{\alpha_{k-1}}).
$$
\end{proof}

\begin{lemma}
\label{homology d b}
The homology of $\partial^{(b)}$ is given by the following equation:
$$\dim H_{k}(\cE_r[U],\partial^{(b)})=
\begin{cases}
\binom{r-1}{k},&\text{if}\  k< b\\
0,& \text{if}\ k\ge b.
\end{cases}
$$
\end{lemma}

\begin{proof}
Since $\partial$ is acyclic, one immediately gets $H_{k}(\cE_r[U],\partial^{(b)})=0$ for $k\ge b$.
For $k<b$, the homology is supported at the zeroth power of $U$ and one has $H_{k}(\cE_r[U])\simeq \Ker(\partial|_{\wedge^k(z_1,\ldots,z_r)})$. 
The dimension of the latter kernel equals
$$
\dim \Ker(\partial|_{\wedge^k(z_1,\ldots,z_r)})=\dim \wedge^{k}(z_1-z_2,\ldots,z_1-z_r)=\binom{r-1}{k}.
$$
\end{proof}



\begin{proof}[Proof of Theorem \ref{homology}]
Let us compute $\HFL(K_{rm,rn},v)$ using the spectral sequence constructed in Theorem \ref{spectral}. 
By Lemma \ref{lem:zero}, in case \ref{it:HFL0} it is easy to see that the complex $(E_1,\partial_1)$ is contractible in the direction of $e_r$ and 
$E_2=H_{*}(E_1,\partial_1)=0$.

In case \ref{it:HFLF} by Lemma \ref{lem:nonzero} and \eqref{d b} one can write $E_1=\cE_{r-\lambda_s}[U]\otimes_{\F[U]} \cE_{\lambda_s}[U]$,
 a tensor product of chain complexes of $\F[U]$--modules, and $\partial_1$ acts as $U\partial$ on the first factor and as $\partial^{(\beta+1)}$ on the second one.
This implies
\begin{equation}
\label{e2minus}
E_2=H_{*}(E_1,\partial_1)\simeq \cE_{r-\lambda_s}\otimes_{\F} H_{*}\left(\cE_{\lambda_s}[U],\partial^{(\beta+1)}\right).
\end{equation}
Indeed, $U$ acts trivially on $H_{*}\left(\cE_{\lambda_s}[U],\partial^{(\beta+1)}\right)$, so one can take the homology of $\partial^{(\beta+1)}$ first and then
observe that $U\partial$ vanishes on
$$
\cE_{r-\lambda_s}[U]\otimes_{\F[U]} H_{*}\left(\cE_{\lambda_s}[U],\partial^{(\beta+1)}\right)\simeq \cE_{r-\lambda_s}\otimes_{\F} H_{*}\left(\cE_{\lambda_s}[U],\partial^{(\beta+1)}\right).
$$
By Lemma \ref{homology d b}, 
the $E_2$ page \eqref{e2minus} agrees with the statement of the theorem, hence we need to prove that the spectral sequence collapses. 

Indeed, the $E_1$ page is bigraded by the homological degree and $|B|$ (see Remark \ref{rem:degrees}). By Lemma \ref{homology d b} any surviving homology class on the $E_2$ page 
of cube degree $x$ has bidegree $(x,-2h_{rm,rn}(v)-2x)$, so all bidegrees on the $E_2$ page belong to the same line of slope $(-2)$. Therefore all higher differentials must vanish.

Finally, a simple formula for $h_{rm,rn}(v)$ in case \ref{it:HFLF} follows from Lemma \ref{lem:h for nonzero}.
\end{proof}

\subsection{Action of \texorpdfstring{$U_i$}{Ui}}

One can use Proposition \ref{U on HFL} to compute the action of $U_i$ on $\HFL$ for cable links. 
Recall that $R=\F[U_1\ldots,U_r]$. Throughout this section we assume $n/m>2g(K)-1$.
We start with a simple algebraic statement.

\begin{proposition}
\label{fixgrad}
Let $\cC$ be an $\F$-algebra. Given a finite collection of elements $c_{\alpha}\in \cC$ and vectors $v^{(\alpha)}\in \Z^r$, consider
the ideal $\cI\subset \cC\otimes_{\F} R$ generated by $c_{\alpha}\otimes U_1^{v^{(\alpha)}_1}\cdots U_r^{v^{(\alpha)}_r}$.
Then the following statements hold:
\begin{enumerate}
\item[(a)] The quotient $(\cC\otimes_{\F} R)/\cI$ can be equipped with a $\Z^r$--grading, with $U_i$ of grading $(-e_i)$ and $\cC$ of grading 0.
\item[(b)] The subspace of $(\cC\otimes_{\F} R)/\cI$ with grading $v$ is isomorphic to
$$
\left[(\cC\otimes_{\F} R)/\cI\right] (v)\simeq \cC/\left(c_{\alpha}: v^{(\alpha)}\preceq -v\right).
$$
\end{enumerate}
\end{proposition}

\begin{proof}
Straightforward.
\end{proof}

\begin{definition}
We define $\cA_r=\cE_r\otimes_{\F} R$ and $\cAr_r=\cEr_r\otimes_{\F} R$.
Let $\cI'_{\beta}$ denote the ideal in $\cA_r$ generated by the monomials $(z_{i_1}\wedge\cdots\wedge z_{i_s})\otimes U_{i_{s+1}}\cdots U_{i_{\beta+1}}$
for all $s\le \beta+1$ and all tuples of pairwise distinct $i_1,\ldots,i_{\beta+1}$. Let $\cI_{\beta}:=\cI'_{\beta}\cap \cAr_r$ be the corresponding ideal
in $\cAr_r$.
\end{definition}

The algebras $\cA_r$ and $\cAr_r$ are naturally $\Z^{r+1}$--graded: the generators $z_i$ have Alexander grading $0$ and homological grading $(-1)$, the generators $U_i$ have Alexander grading $(-e_i)$ and homological grading $(-2)$. 

\begin{definition}
We define $\cH(k):=\bigoplus_{\max(v)\le k}\HFL(K_{rm,rn},v)$. Since $U_i$ decreases the Alexander grading, 
$\cH(k)$ is naturally an $R$--module.
\end{definition}

The following theorem clarifies the algebraic structure of Theorem \ref{homology}.

\begin{theorem}
\label{same max}
The following graded $R$--modules are isomorphic:
$$
\cH(k)/\cH(k-1)\simeq \cAr_r/\cI_{\beta(k)}[-2\hh(k)]\{k,\ldots,k\},
$$
where $[\cdot]$ and $\{\cdot\}$ denote the shifts of the homological grading and the Alexander grading, respectively.
\end{theorem}

\begin{proof}
By definition, $\cH(k)/\cH(k-1)$ is supported on the set of Alexander gradings $v$ such that $\max(v)=k$.
The monomial $U_1\cdots U_r$ belongs to the ideal $\cI_{\beta(k)}$, so $\cAr_r/\cI_{\beta(k)}$ is supported on the set of Alexander gradings $u$ with $\max(u)=0$. 

Suppose that exactly $\lambda$ components of $v$ are equal to $k$. Without loss of generality we can assume
$v_1,\ldots,v_{r-\lambda}<k$ and $v_{r-\lambda+1}=\ldots=v_r=k$. It follows from Lemma \ref{Ker d} and the proof of Theorem \ref{homology} 
that $\HFL(K_{rm,rn},v)$ is isomorphic to the quotient of $\cEr_r$ by the ideal generated by degree $\beta-r+\lambda+1$ monomials
in $(z_i-z_j)$ for $i,j>r-\lambda$.  

Consider the subspace of $\cA_r/\cI'_{\beta}$ of Alexander grading $(v_1-k,\ldots,v_{r}-k)$.
By Proposition \ref{fixgrad} it is isomorphic to a quotient of $\cE_r$ modulo the following relations.
For each subset $B\subset \{1,\ldots,r-\lambda\}$ and each degree $\beta+1-|B|$ monomial $m'$ in 
variables $z_i$ for $i\notin B$ there is a relation $m'\otimes \prod_{b\in B}U_b\in \cI'_{\beta}$.
All these relations can be multiplied by an appropriate monomial in $R$ to have Alexander grading $(v_1-k,\ldots,v_{r}-k)$.

Note that such $m'$ should contain at most $r-\lambda-|B|$ factors with indices in  $\{1,\ldots,r-\lambda\}\setminus B$,
hence it contains at least $\beta-r+\lambda+1$ factors with indices in  $\{r-\lambda+1,\ldots,r\}$. 
Therefore  $\left[\cA_r/\cI'_{\beta}\right](v_1-k,\ldots,v_{r}-k)$ is naturally isomorphic to the quotient of $\cE_r$ by the ideal generated by degree $\beta-r+\lambda+1$ monomials
in $z_i$ for $i>r-\lambda$. 

We conclude that $\left[\cAr_r/\cI_{\beta(k)}\right](v_1-k,\ldots,v_{r}-k)$ is isomorphic to $\HFL(K_{rm,rn},v)$. 
The action of $U_i$ on $\cH(k)$ is described by Proposition \ref{U on HFL}. One can check that it commutes with the 
above isomorphisms for different $v$, so we get the isomorphism of $R$--modules.
\end{proof}

We illustrate the above theorem with the following example (cf. Example \ref{ex:3 3 link}).

\begin{example}
Let us describe the subspaces of $\cAr_3/\cI_{1}$ with various Alexander gradings. The ideal $\cI_1$ 
equals:
$$
\cI_1=\left((z_1-z_2)(z_2-z_3),(z_1-z_2)U_3,(z_1-z_3)U_2,(z_2-z_3)U_1,U_1U_2,U_1U_3,U_2U_3\right)\subset \cAr_3.
$$
In the Alexander grading $(0,0,0)$ one  gets
$$\left[\cAr_3/\cI_{1}\right](0,0,0)\simeq \cEr_3/((z_1-z_2)(z_2-z_3))=\langle 1, z_1-z_2, z_2-z_3\rangle,$$
in the Alexander grading $(k,0,0)$ (for $k>0$) one gets two relations 
$$U_1^{k}(z_1-z_2)(z_2-z_3),U_1^{k-1}(z_2-z_3)\in \cI_1.$$
Since the latter implies the former, we get
$$
\left[\cAr_3/\cI_{1}\right](k,0,0)\simeq \cEr_3/(z_2-z_3)=\langle 1, z_1-z_2\rangle.
$$
The map $U_1:\left[\cAr_3/\cI_{1}\right](0,0,0)\to \left[\cAr_3/\cI_{1}\right](1,0,0)$ is a natural projection
$$
\cEr_3/((z_1-z_2)(z_2-z_3)) \to \cEr_3/(z_2-z_3),
$$
while the map $U_1:\left[\cAr_3/\cI_{1}\right](k,0,0)\to \left[\cAr_3/\cI_{1}\right](k+1,0,0)$
is an isomorphism for $k>0$.

The gradings $(0,k,0)$ and $(0,0,k)$ can be treated similarly. Furthermore, $U_iU_j\in \cI_1$ for $i\neq j$,
so all other graded subspaces of $\cAr_3/\cI_{1}$ vanish. 
\end{example}

Since the multiplication by $U_i$ preserves the ideal $\cI_{\beta}$, we get the following useful result.

\begin{corollary}
\label{surj}
If $\max(v)=\max(v-e_i)$, then the map  $$U_i:\HFL(K_{rm,rn},v)\to \HFL(K_{rm,rn},v-e_i)$$ is surjective.
\end{corollary}

\begin{lemma}
\label{diff max}
Suppose that $\max(v)=k$ and $\max(v-e_i)=k-1$, and the homology group $\HFL(K_{rm,rn},v)$ 
does not vanish. Then $\beta(k)=r-1,\beta(k-1)\ge r-2$ and the map 
$$
U_i:\HFL(K_{rm,rn},v)\to \HFL(K_{rm,rn},v-e_i)
$$ 
is surjective.
\end{lemma}

\begin{proof}
Since $\max(v)=k$ and $\max(v-e_i)=k-1$, the multiplicity of $k$ in $v$ equals 1, so by Theorem \ref{homology} 
$\beta(k)\ge r-1$, hence $\beta(k)=r-1$. Therefore $\HFL(K_{rm,rn},v)\simeq \cEr_r$,
so $U_i$ is surjective. Indeed, by Theorem \ref{homology} $\HFL(K_{rm,rn},v-e_i)$ is naturally isomorphic to a quotient of $\cEr_r$,
and by Proposition \ref{U on HFL} $U_i$ coincides with a natural quotient map.  Finally, by \eqref{beta as max} $\HFK(K_{m,n},k-c+l(r-1))\simeq\F$, and by Lemma \ref{HFL after l} $\HFK(K_{m,n},k-1-c+l(r-2))\simeq\F$, so $\beta(k-1)\ge r-2$.
\end{proof}

\begin{proof}[Proof of Theorem \ref{th:splitting}]
Let us prove that the homology classes with diagonal Alexander gradings generate $\HFL$ over $R$. Indeed, given $v=(v_1\le \ldots\le v_r)$ with $\HFL(K_{rm,rn},v)\neq 0$, by Theorems \ref{homology} and \ref{same max} 
one can check that  $\HFL(K_{rm,rn},v_r,\ldots,v_r)\neq 0$ and by Corollary \ref{surj} the map
$$
U_1^{v_r-v_1}\cdots U_{r-1}^{v_{r}-v_{r-1}}:\HFL(K_{rm,rn},v_r,\ldots,v_r)\to \HFL(K_{rm,rn},v)
$$
is surjective.

Let us describe the $R$-modules generated by the diagonal classes in degree $(k,\ldots,k)$. If $\beta(k)=-1$  then
$\HFL(K_{rm,rn},k,\ldots,k)=0$. If $0\le \beta(k)\le r-2$ then 
by Lemma \ref{diff max} the submodule $R\cdot \HFL(K_{rm,rn},k,\ldots,k)$ does not contain any classes with maximal Alexander degree less than $k$,
so by Theorem \ref{same max}
$$R\cdot \HFL(K_{rm,rn},k,\ldots,k)\simeq \cAr_r/\cI_{\beta(k)}=:M_{\beta(k)}$$

Suppose that $\beta(k)=r-1$, and consider minimal $a$ and maximal $b$ such that $a\le k\le b$ and $\beta(i)=r-1$ for $i\in [a,b]$. If there is no minimal $a$, we set $a=-\infty$. By Lemma \ref{diff max}, $\beta(a-1)=r-2$ and all the maps
\begin{multline*}
\HFL(K_{rm,rn},b,\ldots,b)\stackrel{U_1\cdots U_r}{\longrightarrow} \HFL(K_{rm,rn},b-1,\ldots,b-1)\to \ldots\\
\ldots\to \HFL(K_{rm,rn},a,\ldots,a)\stackrel{U_1\cdots U_r}{\longrightarrow} \HFL(K_{rm,rn},a-1,\ldots,a-1)
\end{multline*}
are surjective. Therefore 
$$
R\cdot \HFL(K_{rm,rn},b,\ldots,b)\simeq \cAr_r/(U_1\cdots U_r)^{b-a}\cI_{r-2}=:M_{r-1,b-a+1}
$$ is supported in all Alexander degrees with maximal coordinates in $[a,b]$ and in Alexander degrees with maximal coordinate $(a-1)$ which appears with multiplicity at least 2. 

Finally, we get the following decomposition of $\HFL$ as an $R$--module:
$$
\HFL(K_{rm,rn})=\bigoplus_{k:0\le \beta(k)<r-1\atop \beta(k+1)<r-1}M_{\beta(k)}\oplus \bigoplus_{a,b:\beta(a-1)=r-2\atop
{\beta(b+1)<r-1\atop \beta([a,b])=r-1}}M_{r-1,b-a+1}\oplus M_{r-1,\infty}.
$$




\end{proof}

Note that for $r=1$ we get $M_{0,l}\simeq \F[U_1]/(U_1^l)$ and $M_{0,+\infty}\simeq \F[U]$.

\subsection{Spectral sequence for \texorpdfstring{$\HFLhat$}{HFL-hat}}

\begin{theorem}
\label{diagonal degeneration}
If $\beta(k)+\beta(k+1)\le r-2$ then the spectral sequence for $\HFLhat(K_{rm,rn},k,\ldots,k)$ degenerates at the $\widehat{E_2}$ page and
$$
\HFLhat(K_{rm,rn},k,\ldots,k)\simeq \bigoplus_{i=0}^{\beta(k)}\binom{r-1}{i}\F_{-2\hh(k)-i}\oplus
\bigoplus_{i=0}^{\beta(k+1)}\binom{r-1}{i}\F_{-2\hh(k)+2-r+i}.
$$
\end{theorem}

\begin{proof}
By Proposition \ref{spectral for hat}, for a given $v$ there is a spectral sequence with $\widehat{E_1}$ page 
$$
\widehat{E_1}=\bigoplus_{B\subset \{1,\ldots,r\}}\HFL(L,v+e_B)
$$ 
and converging to $\widehat{E}_{\infty}=\HFLhat(L,v)$. If $v=(k,\ldots,k)$ then (for $B\neq \emptyset$)
the maximal coordinate of  $v+e_B$ equals $k+1$ and appears with multiplicity $\lambda=|B|$. Therefore, by Theorem \ref{homology}
$\HFL(L,v+e_B)$ does not vanish if and only if either $B=\emptyset$ or $|B|\ge r-\beta(k+1)$, and it is given by Theorem \ref{homology}. By \eqref{def hh} we have $\hh(k+1)=\hh(k)-\beta(k+1)-1.$

The spectral sequence is bigraded by the homological (Maslov) grading at each vertex of the cube and the ``cube grading'' $|B|$. 
The differential $\widehat{\partial_1}$ acts along the edges of the cube, and decreases the Maslov grading by $2$ and the cube grading by 1.

One can check using Theorem \ref{same max}  that its homology $\widehat{E_2}$ does not vanish in cube degrees $0$ and $r-\beta(k+1)$, so one can write $\widehat{E_2}=\widehat{E_2^{0}}\oplus \widehat{E_2^{r-\beta(k+1)}},$ and
$$
\widehat{E_2^{0}}\simeq \bigoplus_{i=0}^{\beta(k)}\binom{r-1}{i}\F_{-2\hh(k)-i},\qquad
\widehat{E_2^{r-\beta(k+1)}}\simeq \bigoplus_{i=0}^{\beta(k+1)}\binom{r-1}{i}\F_{-2\hh(k+1)-3\beta(k+1)+i}.
$$
By \eqref{def hh} we have $\hh(k+1)=\hh(k)-\beta(k+1)-1,$ so $-2\hh(k+1)-3\beta(k+1)+i=-2\hh(k)+2-\beta(k+1)+i$.

A higher  differential $\widehat{\partial_s}$ decreases the cube grading by $s$ and decreases the Maslov grading by $s+1$. 
Therefore the only nontrivial higher differential is $\widehat{\partial_{r-\beta(k+1)}}$ which vanishes by degree reasons too. Indeed, the maximal Maslov grading in $\widehat{E_2^{r-\beta(k+1)}}$ equals $-2\hh(k)+2$ while the minimal Maslov grading in 
 $\widehat{E_2^{0}}$ equals $-2\hh(k)-\beta(k)$, so the differential can decrease the Maslov grading at most by $\beta(k)+2$.
On the other hand, $\widehat{\partial_{r-\beta(k+1)}}$ drops it by $r-\beta(k+1)+1$, and  for $\beta(k)+\beta(k+1)<r-1$ one has $r-\beta(k+1)+1>\beta(k)+2$.  Therefore $\widehat{\partial_{r-\beta(k+1)}}=0$ and the spectral sequence vanishes at the $\widehat{E_2}$ page.
\end{proof}

We illustrate the proof of Theorem \ref{diagonal degeneration} by Examples \ref{ex:degeneration a} and \ref{ex:degeneration b}

\begin{lemma}
\label{symmetry for beta}
The following identity holds:
$$
\beta(1-k)+\beta(k)=r-2.
$$
\end{lemma}

\begin{proof}
By \eqref{def hh} and Lemma \ref{h on diagonal} we have
$$\beta(k)=h(k-1,\ldots,k-1)-h(k,\ldots,k)-1,\ \beta(1-k)=h(-k,\ldots,-k)-h(1-k,\ldots,1-k)-1.$$
By Lemma \ref{symmetry for h} we have
$$h(-k,\ldots,-k)=h(k,\ldots,k)+kr,\ h(1-k,\ldots,1-k)=h(k-1,\ldots,k-1)+r(k-1).$$
These two identities imply the desired statement.
\end{proof}

\begin{theorem}
\label{diagonal dual}
If $\beta(k)+\beta(k+1)\ge r-2$ then:
$$
\HFLhat(K_{rm,rn},k,\ldots,k)\simeq \bigoplus_{i=0}^{r-2-\beta(k+1)}\binom{r-1}{i}\F_{-2\hh(k)-i}\oplus
\bigoplus_{i=0}^{r-2-\beta(k)}\binom{r-1}{i}\F_{-2\hh(k)+2-r+i}
$$
\end{theorem}

\begin{proof}
By Lemma \ref{symmetry for beta} we get $\beta(-k)=r-2-\beta(k+1)$ and $\beta(1-k)=r-2-\beta(k)$, so
$$
\beta(k)+\beta(k+1)+\beta(-k)+\beta(1-k)=2(r-2),
$$
so  $\beta(-k)+\beta(1-k)\le r-2$.
By Theorem \ref{diagonal degeneration} the spectral sequence degenerates for $\HFLhat(-k,\ldots,-k)$ and 
$$
\HFLhat(K_{rm,rn},-k,\ldots,-k)\simeq \bigoplus_{i=0}^{r-2-\beta(k+1)}\binom{r-1}{i}\F_{-2\hh(-k)-i}\oplus
\bigoplus_{i=0}^{r-2-\beta(k)}\binom{r-1}{i}\F_{-2\hh(-k)+2-r+i}
$$
Finally, by \cite[Proposition 8.2]{OSlinks} we have
$$
\HFLhat_{\bullet}(K_{rm,rn},k,\ldots,k)=\HFLhat_{\bullet-2kr}(K_{rm,rn},-k,\ldots,-k)
$$
and by  Lemma \ref{symmetry for h} $\hh(k)=\hh(-k)-kr$.
\end{proof}

\begin{theorem}
\label{off diagonal}
Off-diagonal homology groups are supported on the union of the unit cubes along the diagonal.
In such a cube with corners $(k,\ldots,k)$ and $(k+1,\ldots,k+1)$ one has 
$$\HFLhat(K_{rm,rn},(k-1)^{j},k^{r-j})\simeq \binom{r-2}{\beta(k)}\F_{-2\hh(k)-\beta(k)-j}.$$
\end{theorem}

\begin{proof}
We use the spectral sequence from $\HFL$ to $\HFLhat$. 
By Theorem \ref{same max}, all the $\widehat{E_2}$ homology outside the union of these cubes vanish (since some $U_i$ would provide an isomorphism 
between $\HFL(K_{rm,rn},v)$ and $\HFL(K_{rm,rn},v-e_i)$).
Furthermore, if $\beta(k)=r-1$ then the homology in the cube vanish too, so we can focus on the case $\beta(k)\le r-2$.

One can check that $\widehat{E_2}$ does not vanish in cube degrees $j-\beta(k),\ldots,j$ and 
$$
\widehat{E_2^{j-c}}\simeq \binom{j-1}{c}\binom{r-1-j}{\beta(k)-c}\F_{-2\hh(k)-\beta(k)-c}.
$$
Note that the {\em total} homological degree on $\widehat{E_2^{j-c}}$ equals $-2\hh(k)-\beta(k)-j$ and does not depend on $c$.
Therefore all higher differentials in the spectral sequence must vanish and the rank of $\HFLhat$ equals:
$$
\sum_{c=0}^{\beta}\binom{j-1}{c}\binom{r-1-j}{\beta(k)-c}=\binom{r-2}{\beta(k)}.
$$
\end{proof}

We illustrate this proof by Example \ref{ex:off diagonal}.
  
\subsection{Special case: \texorpdfstring{$m=1,\ n=2g(K)-1$}{m=1, n=2g(K)-1}}

The case $m=1, n=2g(K)-1$ is special since Lemma \ref{HFL after l} is not always true.
Indeed, $K_{m,n}=K$ and $l=n=2g(K)-1$, but for $v=g(K)-l=1-g(K)$ we have $\HFL(K,v)=0$.
However, it is clear that in all other cases Lemma \ref{HFL after l} is true, so for generic $v$ Lemmas \ref{lem:zero} and
\ref{lem:nonzero} hold true. This allows one to prove an analogue of Theorem \ref{homology}.   

\begin{theorem}
\label{homology special}
Assume that $m=1,n=2g(K)-1$ (so $l=2g(K)-1$) and suppose that $v=(u_1^{\lambda_1},u_2^{\lambda_2},\ldots,u_s^{\lambda_s})$ where $u_1<\ldots <u_s$.
Then the Heegaard-Floer homology group $\HFL(K_{rm,rn},v)$ can be described as following:
\begin{enumerate}[label=(\alph*)]
\item \label{it:special} Assume that $u_s-c+l(r-\lambda_s)=g(K)-\nu l$ with $1\le \nu\le \lambda_s$. Then
$$
\HFL(K_{rm,rn},v)\simeq (\F_{(0)}\oplus \F_{(-1)})^{r-\lambda_s}\otimes \left[\bigoplus_{j=0}^{\nu-2}\binom{\lambda_s-1}{j}\F_{(-2h(v)-j)}\oplus\binom{\lambda_s-1}{\nu}\F_{(-2h(v)+2-\nu)}\right]$$

\item \label{it:generic} In all other cases, the homology is given by Theorem \ref{homology}.
\end{enumerate}
\end{theorem}

\begin{proof}
One can check that the proof of Lemma \ref{lem:zero} fails if $u_s-c+l(r-\lambda_s)=g(K)-l$, 
and remains true in all other cases. Similarly, the proof of Lemma \ref{lem:nonzero} fails  only if $u_s-c+l(r-\lambda_s)+lj=g(K)-l$
for $1\le j\le \lambda_s-1$, which is equivalent to $u_s-c+l(r-\lambda_s)=g(K)-(j+1)$. This proves \ref{it:generic}.

Let us consider the special case \ref{it:special}. Note that
$$
h_{m,n}(u_s-c+l(r-\lambda_s)+lj-1)-h_{m,n}(u_s-c+l(r-\lambda_s)+lj)=$$ $$\chi(\HFK(K,g(K)+l(j-\nu))=
\begin{cases}
1, & \text{if}\ j<\nu-1\\
0, & \text{if}\ j=\nu-1\\
1, & \text{if}\ j=\nu \\
0, &  \text{if}\ j>\nu .\\
\end{cases}
$$
Given a pair of subsets $B'\subset \{1,\ldots,r-\lambda_s\}$ and $B''\subset \{r-\lambda_s+1,\ldots,r\}$, one can
write, analogously to Lemma \ref{lem:nonzero}: 
$$h_{rm,rn}(v-e_{B'}-e_{B''})=h_{rm,rn}(v)+|B'|+w(B''),$$
where 
$$
w(B'')=\begin{cases}
|B''|, & \text{if}\ |B''|\le \nu-1\\
\nu-1, & \text{if}\ |B''|=\nu \\
\nu, & \text{if}\ |B''|>\nu.\\
\end{cases}
$$
By the K\"unneth formula, the $E_2$ page of the spectral sequence is determined by the ``deformed cube homology'' with the weight function $w(B'')$,
as in \eqref{d b}. If $\partial$, as above, denotes the standard cube differential, then, similarly to Lemma \ref{homology d b}, the homology of $\partial^{w}_U$ is isomorphic to the kernel of $\partial$
in cube degrees $0,\ldots \nu-2$ and $\nu.$ 

Finally, we need to prove that all higher differentials vanish. For a homology generator $\alpha$ on the $E_2$ page of cube degree $x$, its bidegree is equal either to 
$(x,-2h(v)-2x)$  or to $(x,-2h(v)-2x+2)$. The differential $\partial_k$ has bidegree $(-k,k-1)$ (see Remark \ref{rem:degrees}), so the bidegree of $\partial_k(\alpha)$
is equal either to  $(x-k,-2h(v)-2x+k-1)$  or to $(x-k,-2h(v)-2x+k+1)$. Since $-2x+k+1<-2(x-k)$ for $k>1$, we have $\partial_k(\alpha)=0$.
\end{proof}

The action of $U_i$ in this special case can be described similarly to Theorem \ref{same max}. However, it is not true that $U_i$ is surjective whenever it does not obviously vanish.
In particular, the following example shows that $\HFL$ may be not generated by diagonal classes, so Theorem \ref{th:splitting} does not hold. We leave the appropriate adjustment of 
Theorem \ref{th:splitting} as an exercise to a reader.

\begin{example}
Consider $T_{2,2}$, the $(2,2)$ cable of the trefoil. We have $g(K)=l=1$ and $c=1/2$, so by Theorem \ref{homology special}
$$
\HFL(T_{2,2},1/2,1/2)\simeq \F_{(-1)},\qquad \HFL(T_{2,2},-1/2,1/2)\simeq \F_{(-2)}\oplus \F_{(-3)}.
$$
Therefore $U_1$ is not surjective. Furthermore, the class in $\HFL(T_{2,2},-1/2,1/2)$ of homological degree $(-2)$
is not in the image of any diagonal class under the $R$--action.
\end{example} 

\section{Examples} \label{sec:examples}


\subsection{\texorpdfstring{$(n,n)$}{(n,n)} torus links}

The symmetrized multi-variable Alexander polynomial of the $(n,n)$ torus link equals (for $n>1$):
$$
\Delta_{T_{n,n}}(t_1,\ldots,t_n)=((t_1\cdots t_n)^{1/2}-(t_1\cdots t_n)^{-1/2})^{n-2}.
$$
Each pair of components has linking number 1, so $c=(n-1)/2$. 
The homology groups $\HFL(T(n,n),v)$ are described by the following theorem, which is a special case of Theorem \ref{homology}.

\begin{theorem}
\label{th: n n minus}
Consider the $(n,n)$ torus link, and an Alexander grading $v=(v_1,\ldots,v_n)$. Suppose that among the coordinates $v_i$ exactly $\lambda$ are equal to $k$ and all other coordinates are less than $k$. Let $|v|=v_1+\ldots+v_n$. Then
$$
\HFL(T(n,n),v)=\begin{cases}
0& \text{if}\ k>\lambda-\frac{n+1}{2},\\
(\F_{(0)}\oplus \F_{(-1)})^{n-1}\otimes \F_{2|v|}& \text{if}\ k<-\frac{n-1}{2},\\
(\F_{(0)}\oplus \F_{(-1)})^{n-\lambda}\otimes \bigoplus_{i=0}^{\lambda-\frac{n+1}{2}-k} \binom{\lambda-1}{i}\F_{(-2h(v)-i)}& \text{if}\  -\frac{n-1}{2}\le k\le \lambda-\frac{n+1}{2}, \\
\end{cases}
$$
where
$h(v)=\frac{1}{2}(\frac{n-1}{2}-k)(\frac{n-1}{2}-k+1)+kn-|v|$ in the last case.
\end{theorem}

\begin{proof}
Indeed, $\beta(k)=\frac{n-1}{2}-k$ for $k>-\frac{n-1}{2}$ and $\beta(k)=n-1$ for $k\le -\frac{n-1}{2}$.
By Theorem \ref{homology}, the homology group $\HFL(T(n,n),v)$ does not vanish if and only if
\begin{equation}
\label{n n semigroup}
k\le \lambda-\frac{n+1}{2}.
\end{equation}
If $k\ge -\frac{n-1}{2}$, equation \eqref{h for link} implies:
$$
h_{n,n}(v)=\frac{1}{2}\left(\frac{n-1}{2}-k\right)\left(\frac{n-1}{2}-k+1\right)+kn-|v|.
$$
If $k\le -\frac{n-1}{2}$, equation \eqref{h for link} implies $h_{n,n}(v)=-|v|$.
Furthermore, for all $v$ satisfying \eqref{n n semigroup} one has
$$
\HFL(T(n,n),v)=(\F_{(0)}\oplus \F_{(-1)})^{n-\lambda}\otimes \bigoplus_{j=0}^{\lambda-\frac{n+1}{2}-k}\binom{\lambda-1}{j}\F_{(-2h_{n,n}(v)-j)}.
$$
Finally, if $k-\frac{n-1}{2}$, then \eqref{n n semigroup} holds for all $\lambda$ and $\lambda-\frac{n+1}{2}-k>\lambda-1$, hence
$$
\HFL(T(n,n),v)=(\F_{(0)}\oplus \F_{(-1)})^{n-\lambda}\otimes \bigoplus_{j=0}^{\lambda-1}\binom{\lambda-1}{j}\F_{(-2h_{n,n}(v)-j)}=
(\F_{(0)}\oplus \F_{(-1)})^{n-1}\otimes \F_{(-2h_{n,n}(v))}.
$$
\end{proof}

\begin{remark}
One can check that, in agreement with \cite{gn}, the condition \eqref{n n semigroup} defines the multi-dimensional semigroup of the plane curve singularity $x^n=y^n$.
\end{remark}

\begin{corollary}
We have the following decomposition of $\HFL$ as an $R$-module: 
$$
\HFL(T(n,n))=M_{0}\oplus M_1\oplus M_2\oplus\ldots\oplus M_{n-2}\oplus M_{n-1,+\infty}.
$$
\end{corollary}

\noindent To prove Theorem \ref{th: n n hat}, we use Theorem \ref{hat homology}. 

\begin{proof}[Proof of Theorem \ref{th: n n hat}]
We have $\beta(\frac{n-1}{2}-s)=s$ for $s<n-1$, and  
$$
\beta(\frac{n-1}{2}-s)+\beta(\frac{n-1}{2}-s+1)=2s-1\le n-2\ \le s\le \frac{n-1}{2}.
$$
Therefore for $s\le \frac{n-1}{2}$ Theorem \ref{diagonal degeneration} implies the degeneration of the spectral sequence
from $\HFL$ to $\HFLhat$, and 
$$
\HFLhat\left(T(n,n),\frac{n-1}{2}-s,\ldots,\frac{n-1}{2}-s\right)=\bigoplus_{i=0}^{s} \binom{n-1}{i}\F_{(-s^2-s-i)}\oplus \bigoplus_{i=0}^{s-1} \binom{n-1}{i}\F_{(-s^2-s-n+2+i)}.
$$
\end{proof}

\noindent Let us illustrate the degeneration of the spectral sequence from $\HFL$ to $\HFLhat$ in some examples.

\begin{example}
\label{ex:degeneration a}
For $s=0$ we have $\widehat{E_1}=\widehat{E_2}=\F_{(0)}$. For $s=1$ the $\widehat{E_1}$ page has nonzero entries in cube degree $0$
where one gets 
$$\HFL\left(T(n,n),\frac{n-1}{2}-1,\ldots,\frac{n-1}{2}-1\right)\simeq \F_{(-2)}\oplus (n-1)\F_{(-3)},$$
and in cube degree $n$ where one gets $\F_{(0)}.$ Indeed, the differential $\widehat{\partial_1}$ vanishes, so for $n>2$
$$
\HFLhat \left(T(n,n),\frac{n-1}{2}-1,\ldots,\frac{n-1}{2}-1\right)\simeq \F_{(-2)}\oplus (n-1)\F_{(-3)}\oplus \F_{(-n)}.
$$
Note that for $n=2$ the differential $\widehat{\partial_{2}}$ does not vanish, so the bound $s\le \frac{n-1}{2}$ is indeed necessary for the spectral sequence to collapse at $\widehat{E_2}$ page. 
\end{example}

\begin{example}
\label{ex:degeneration b}
The case $s=2$ is more interesting. The $\widehat{E_1}$ page has nonzero entries in cube degree $0$, $n-1$ (where we have $n$ vertices) and $n$, where one has
$$
\widehat{E_1^{0}}=\F_{(-6)}\oplus(n-1)\F_{(-7)}\oplus\binom{n-1}{2}\F_{(-8)},\ \widehat{E_1^{n-1}}=n(\F_{(-4)}\oplus \F_{(-5)}),\ \widehat{E_1^{n}}=\F_{(-2)}\oplus (n-1)\F_{(-3)}.
$$
The differential $\widehat{\partial_1}$  cancels some summands in $\widehat{E_1^{n-1}}$ and $\widehat{E_1^{n}}$;
$$
\widehat{E_2^{0}}=\F_{(-6)}\oplus(n-1)\F_{(-7)}\oplus\binom{n-1}{2}\F_{(-8)},\ \widehat{E_2^{n-1}}=(n-1)\F_{(-4)}+\F_{(-5)}.
$$
For $n>4$ all higher differentials vanish and
\begin{multline}
\HFLhat\left(T(n,n),\frac{n-1}{2}-2,\ldots,\frac{n-1}{2}-2\right)\simeq\\ \F_{(-6)}\oplus(n-1)\F_{(-7)}\oplus\binom{n-1}{2}\F_{(-8)}\oplus (n-1)\F_{(-3-n)}+\F_{(-4-n)}.
\end{multline}
\end{example}

\noindent The following example illustrates the computation of $\HFLhat$ for the off-diagonal Alexander gradings.

\begin{example}
\label{ex:off diagonal}
Let us compute the homology $\HFLhat(T(n,n),v)$ for $v=(\frac{n-1}{2}-2)^{j}(\frac{n-1}{2}-1)^{n-j}$ ($1\le j\le n-1$) using the spectral sequence from $\HFL$.
In the $n$ dimensional cube $(v+e_B)$ almost all all vertices have vanishing $\HFL$, except for the vertex $(\frac{n-1}{2}-1,\ldots,\frac{n-1}{2}-1)$
$$
\HFL(\frac{n-1}{2}-1,\ldots,\frac{n-1}{2}-1)=F_{(-2)}\oplus (n-1)\F_{(-3)}
$$
and $j$ of its neighbors with homology $\F_{(-4)}\oplus \F_{(-5)}$.
Clearly, $\widehat{E_2}$ is concentrated in degrees $j$ (with homology $(n-1-j)\F_{(-3)}$) and
$(j-1)$ (with homology $(j-1)\F_{(-4)})$. Note that both parts contribute to the total degree $(-3-j)$,
so
$$
\HFLhat(T(n,n),v)=(n-1-j)\F_{(-3-j)}\oplus (j-1)\F_{(-3-j)}=(n-2)\F_{(-3-j)}.
$$ 
\end{example}

\noindent Finally, we draw all the homology groups $\HFL$ for $(2,2)$ and $(3,3)$ torus links.

\begin{example}
\label{ex:2 2 link}
For the Hopf link, one has two cases. If $v_1<v_2$, then the condition \eqref{n n semigroup} implies $v_2\le -1/2$. 
If $v_1=v_2$, then \eqref{n n semigroup} implies $v_2\ge 1/2$.

The nonzero homology of the Hopf link is shown in Figure \ref{fig:2 2 link} and Table \ref{tab:2 2 link}
\end{example}

\begin{figure}[ht]
\begin{tikzpicture}
\draw[dashed,->] (0,0)--(-6,0) node [below] {$v_1$};
\draw[dashed,->] (0,0)--(0,-6) node [right] {$v_2$};
 
\draw [step=1,dashed]  (-5,-5) grid (0,0);

\fill[fill=lightgray] (-1,-1)--(-5,-1)--(-5,-5)--(-1,-5)--(-1,-1);
\draw[ultra thick] (-5,-1)--(-1,-1)--(-1,-5);

\draw (0,0) node [circle,draw=black,fill=white] {$\F$};
\draw (-1,-1) node [circle,draw=black,fill=white] {$\F^2$};

\draw (0,0.7) node {$\frac{1}{2}$};
\draw (-1,0.5) node {$-\frac{1}{2}$};
\draw (-2,0.5) node {$-\frac{3}{2}$};
\draw (-3,0.5) node {$-\frac{5}{2}$};
\draw (-4,0.5) node {$-\frac{7}{2}$};

\draw (0.7,0) node {$\frac{1}{2}$};
\draw (0.5,-1) node {$-\frac{1}{2}$};
\draw (0.5,-2) node {$-\frac{3}{2}$};
\draw (0.5,-3) node {$-\frac{5}{2}$};
\draw (0.5,-4) node {$-\frac{7}{2}$};
\end{tikzpicture}
\caption{$\HFL$ for the (2,2) torus link: $\F^2$  on thick lines and in the grey region}
\label{fig:2 2 link}
\end{figure}
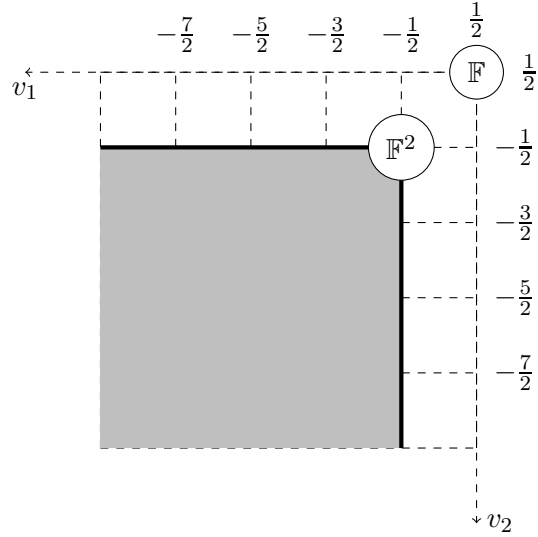

\begin{table}[ht]
\begin{tabular}{|c|c|}
\hline
Alexander grading & Homology\\
\hline
$(1/2,1/2)$ & $\F_{(0)}$\\
\hline
$(a,b)$, $a,b\le -1/2$& $\F_{(2a+2b)}\oplus \F_{(2a+2b-1)}$\\
\hline
\end{tabular}
\caption{Maslov gradings for the $(2,2)$ torus link}
\label{tab:2 2 link}
\end{table}

\begin{example}
\label{ex:3 3 link}
For the $(3,3)$ torus link, one has two cases. If $v_1\le v_2<v_3$, then the condition \eqref{n n semigroup} implies $v_3\le 1$. 
If $v_1<v_2=v_3$, then \eqref{n n semigroup} implies $v_3\le 0$. Finally, if $v_1=v_2=v_3$, then \eqref{n n semigroup} implies $v_3\le 1$.
In other words, nonzero homology appears at the point $(1,1,1)$, at three lines $(0,0,k),(0,k,0),(k,0,0)$ $(k\le 0)$ and at the octant $\max(v_1,v_2,v_3)\le -1$.

This homology is shown in Figure \ref{fig:3 3 link} and Table \ref{tab:3 3 link}.
\end{example}

\begin{figure}[ht]
\tdplotsetmaincoords{40}{-20}
\begin{tikzpicture}[tdplot_main_coords]
\draw[dashed,->] (0,0,0) -- (-10,0,0) node[anchor=north east]{$v_1$};
\draw[dashed,->] (0,0,0) -- (0,-15,0) node[anchor=north west]{$v_2$};
\draw[dashed,->] (0,0,0) -- (0,0,-20) node[anchor=south]{$v_3$};

\draw[dashed] (-2,0,0)--(-2,-2,0)--(0,-2,0);
\draw[dashed] (0,0,-2)--(-2,0,-2)--(-2,-2,-2)--(0,-2,-2)--(0,0,-2);
\draw[dashed] (-2,0,0)--(-2,0,-2);
\draw[dashed] (0,-2,0)--(0,-2,-2);
\draw[dashed] (-2,-2,0)--(-2,-2,-2);

\draw[dashed] (-2,-4,-2)--(-2,-2,-2)--(-4,-2,-2)--(-4,-4,-2); 
\draw[dashed] (-4,-2,-4)--(-4,-4,-4)--(-2,-4,-4)--(-2,-2,-4);
\draw[dashed] (-2,-2,-2)--(-2,-2,-4);
\draw[dashed] (-2,-4,-2)--(-2,-4,-4);
\draw[dashed] (-4,-2,-2)--(-4,-2,-4);
\draw[dashed] (-4,-4,-2)--(-4,-4,-4);
\draw[dashed] (-4,-2,-4)--(-2,-2,-4);

\draw [ultra thick,dashed] (-4,-2,-2)--(-10,-2,-2);
\draw[ultra thick,dashed] (-2,-4,-2)--(-2,-13,-2);
\draw [ultra thick,dashed] (-2,-2,-4)--(-2,-2,-20);

\fill[fill=lightgray] (-4,-4,-4)--(-10,-4,-4)--(-10,-4,-10)--(-4,-4,-10)--(-4,-4,-4);
\fill[fill=lightgray] (-10,-4,-4)--(-10,-4,-10)--(-10,-10,-10)--(-10,-10,-4)--(-10,-4,-4);
\fill[fill=lightgray] (-4,-4,-10)--(-10,-4,-10)--(-10,-10,-10)--(-4,-10,-10)--(-4,-4,-10);
\fill[fill=lightgray] (-4,-4,-4)--(-4,-10,-4)--(-4,-10,-10)--(-4,-4,-10)--(-4,-4,-4);
\draw[ultra thick] (-4,-4,-4)--(-10,-4,-4);
\draw[ultra thick] (-4,-4,-4)--(-4,-10,-4);
\draw[ultra thick] (-4,-4,-4)--(-4,-4,-10);
\draw  (-4,-10,-4)--(-4,-10,-10);
\draw [dashed] (-4,-4,-10)--(-4,-10,-10);
\draw [dashed] (-4,-4,-10)--(-10,-4,-10);

\draw (0,0,0) node [circle,draw=black,fill=white] {$\F$};
\draw (-2,-2,-4) node [circle,draw=black,fill=white]  {$\F^2$};
\draw (-2,-2,-2) node [circle,draw=black,fill=white] {$\F^3$};
\draw (-4,-4,-4) node [circle,draw=black,fill=white]  {$\F^4$};
\draw[dashed] (-4,-4,-2)--(-2,-4,-2);
\draw (-2,-4,-2) node [circle,draw=black,fill=white] {$\F^2$};
\draw (-4,-2,-2) node [circle,draw=black,fill=white]  {$\F^2$};

\end{tikzpicture}
\caption{$\HFL$ for the (3,3) torus link: $\F^2$  on dashed thick lines;  $\F^4$ on solid thick lines and in the shaded region.
Top Alexander grading is $(1,1,1)$.}
\label{fig:3 3 link}
\end{figure}
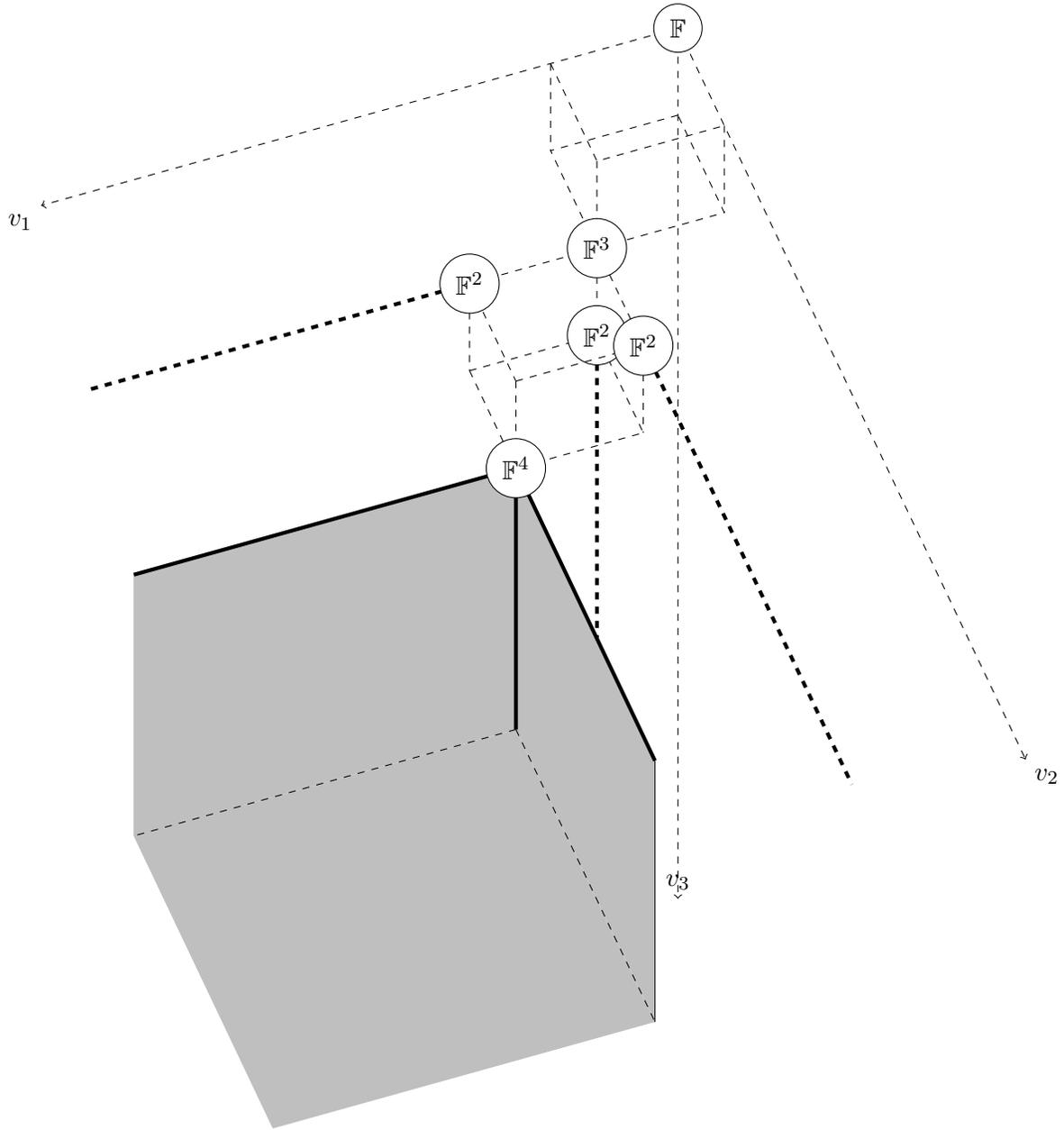

\begin{table}[ht]
\begin{tabular}{|c|c|}
\hline
Alexander grading & Homology\\
\hline
$(1,1,1)$ & $\F_{(0)}$ \\
\hline
$(0,0,0)$ & $\F_{(-2)}\oplus 2\F_{(-3)}$ \\
\hline
$(0,0,k)$, $(0,k,0)$ and $(k,0,0)$ ($k<0$)& $\F_{(2k-2)}\oplus \F_{(2k-3)}$\\
\hline
$(a,b,c)$, $a,b,c\le -1$ & $\F_{(2a+2b+2c)}\oplus 2\F_{(2a+2b+2c-1)}\oplus \F_{(2a+2b+2c-2)}$\\
\hline
\end{tabular}
\caption{Maslov gradings for the $(3,3)$ torus link}
\label{tab:3 3 link}
\end{table}

\subsection{More general torus links}

The $\HFL$ homology of the $(4,6)$ torus link is shown in Figure \ref{fig:4 6 link} and Table \ref{tab:4 6 link}.
Note that as an $\F[U_1,U_2]$ module it can be decomposed into 5 copies of $M_0\simeq \F$, a copy of $M_{1,1}$ and a copy of $M_{1,+\infty}$. In particular, the map $U_1U_2:\HFL(-2,-2)\to \HFL(-3,-3)$ is surjective with one-dimensional kernel.

\begin{figure}[ht]
\begin{tikzpicture}
\draw[dashed,->] (0,0)--(-11,0) node [below] {$v_1$};
\draw[dashed,->] (0,0)--(0,-11) node [right] {$v_2$};
 
\draw [step=1,dashed] (-10,-10) grid (0,0) ;

\fill[fill=lightgray] (-8,-8)--(-8,-10)--(-10,-10)--(-10,-8)--(-8,-8);
\draw[ultra thick] (-10,-8)--(-8,-8)--(-8,-10);

\draw[ultra thick] (-10,-6)--(-7,-6);
\draw[ultra thick] (-6,-7)--(-6,-10);

\draw (-0,-0) node [circle,draw=black,fill=white] {$\F$};
\draw (-2,-2) node [circle,draw=black,fill=white] {$\F$};
\draw (-3,-3) node [circle,draw=black,fill=white] {$\F$};
\draw (-4,-4) node [circle,draw=black,fill=white] {$\F$};
\draw (-5,-5) node [circle,draw=black,fill=white] {$\F$};
\draw (-6,-6) node [circle,draw=black,fill=white] {$\F^2$};
\draw (-7,-7) node [circle,draw=black,fill=white] {$\F$};
\draw (-8,-8) node [circle,draw=black,fill=white] {$\F^2$};

\draw  (-6,-7) node [circle,draw=black,fill=white] {$\F^2$};
\draw  (-7,-6) node [circle,draw=black,fill=white] {$\F^2$};
 
\draw (0,0.7) node {$4$}; 
\draw (-1,0.7) node {$3$}; 
\draw (-2,0.7) node {$2$}; 
\draw (-3,0.7) node {$1$}; 
\draw (-4,0.7) node {$0$}; 
\draw (-5,0.7) node {$-1$}; 
\draw (-6,0.7) node {$-2$}; 
\draw (-7,0.7) node {$-3$}; 
\draw (-8,0.7) node {$-4$}; 
\draw (-9,0.7) node {$-5$};  

\draw (0.7,0) node {$4$}; 
\draw (0.7,-1) node {$3$}; 
\draw (0.7,-2) node {$2$}; 
\draw (0.7,-3) node {$1$}; 
\draw (0.7,-4) node {$0$}; 
\draw (0.6,-5) node {$-1$}; 
\draw (0.6,-6) node {$-2$}; 
\draw (0.6,-7) node {$-3$}; 
\draw (0.6,-8) node {$-4$}; 
\draw (0.6,-9) node {$-5$};

\end{tikzpicture}
\caption{$\HFL$ for the (4,6) torus link: $\F^2$  on thick lines and in the grey region}
\label{fig:4 6 link}
\end{figure}
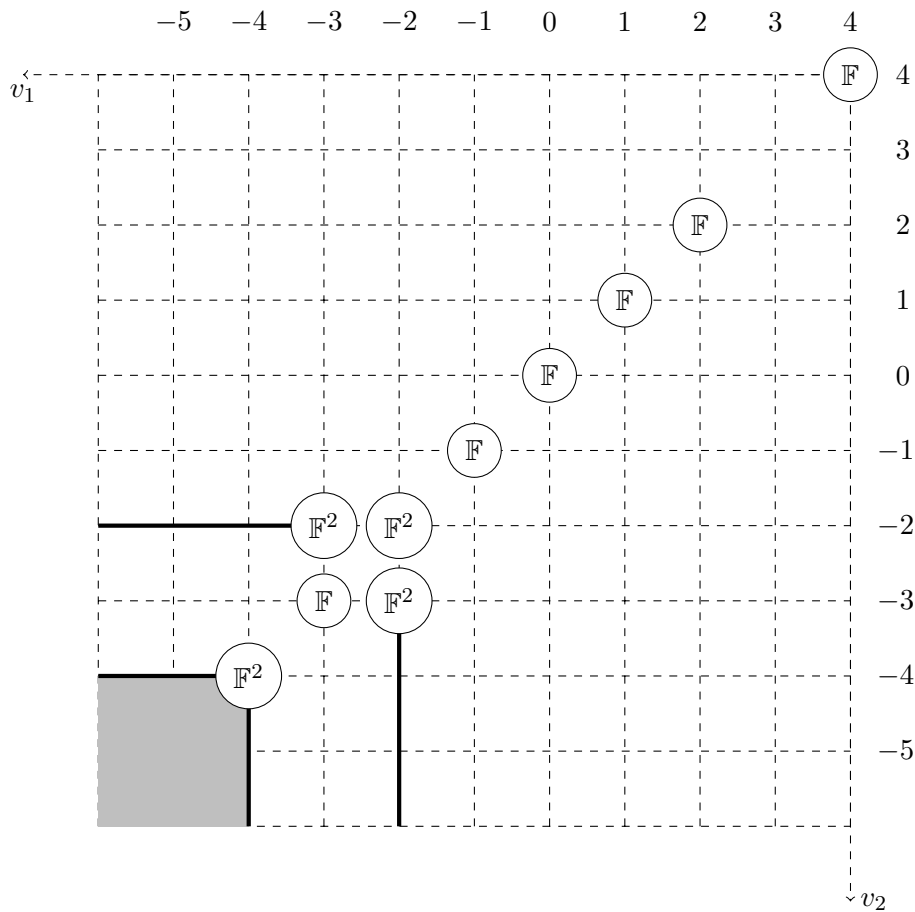

\begin{table}[ht]
\begin{tabular}{|c|c|}
\hline
Alexander grading & Homology\\
\hline
$(4,4)$ & $\F_{(0)}$ \\
\hline
$(2,2)$ & $\F_{(-2)}$\\
\hline
$(1,1)$ & $\F_{(-4)}$\\
\hline
$(0,0)$ & $\F_{(-6)}$ \\
\hline
$(-1,-1)$ & $\F_{(-8)}$ \\
\hline
$(-2,k)$ and $(k,-2)$, $k\le -2$ & $\F_{(2k-6)}\oplus \F_{(2k-7)}$\\
\hline
$(-3,-3)$ & $\F_{(-12)}$ \\
\hline
$(a,b)$, $a,b\le -4$ & $\F_{(2a+2b)}\oplus \F_{(2a+2b-1)}$\\
\hline
\end{tabular}
\caption{Maslov gradings for the $(4,6)$ torus link}
\label{tab:4 6 link}
\end{table}

\subsection{Non-algebraic example}

In this subsection we compute the Heegaard-Floer homology for the $(4,6)$-cable of the trefoil. Its components are $(2,3)$-cables of the trefoil,
which are known to be L-space knots (cf. \cite{HeddencablingII}), but not algebraic knots.  By Theorem \ref{thm:cablelink}, the $(4,6)$-cable of the trefoil is an L-space link,
but its homology is not covered by \cite{gn}.

The Alexander polynomial of the $(2,3)$-cable of the trefoil equals:
$$
\Delta_{T_{2,3}}(t)=\frac{(t^6-t^{-6})(t^{1/2}-t^{-1/2})}{(t^{3/2}-t^{-3/2})(t^2-t^{-2})},
$$
hence the Euler characteristic of its Heegaard-Floer homology equals
$$
\chi_{2,3}(t)=\frac{\Delta_{T_{2,3}}(t)}{1-t^{-1}}=t^3+1+t^{-1}+t^{-3}+t^{-4}+\ldots
$$
By \eqref{alexander cabling}, the bivariate Alexander polynomial of the $(4,6)$-cable equals:
$$
\chi_{4,6}(t_1,t_2)=\chi_{2,3}(t_1\cdot t_2)((t_1t_2)^{3}-(t_1t_2)^{-3})
$$
$$
=(t_1t_2)^{6}+(t_1t_2)^{3}+(t_1t_2)^{2}+(t_1t_2)^{-1}+(t_1t_2)^{-2}+(t_1t_2)^{-5}.
$$
The nonzero Heegaard-Floer homology are shown in Figure \ref{fig:4 6 cable of trefoil} and the corresponding Maslov gradings are given in Table \ref{tab:4 6 cable of trefoil}. Note that as $\F[U_1,U_2]$ module it can be decomposed in the following way: 
$$
\HFL\simeq 4M_{0}\oplus M_{1,1}\oplus M_{1,2}\oplus M_{1,+\infty}.
$$

\bibliographystyle{amsalpha}

\bibliography{bib}

\begin{table}[!ht]
\begin{tabular}{|c|c|}
\hline
Alexander grading & Homology\\
\hline
$(6,6)$ & $\F_{(0)}$ \\
\hline
$(3,3)$ & $\F_{(-2)}$ \\
\hline
$(2,2)$ & $\F_{(-4)}$ \\
\hline
$(0,k)$ and $(k,0)$, $k\ge 0$ & $\F_{(2k-6)}\oplus \F_{(2k-7)}$\\
\hline
$(-1,-1)$ & $\F_{(-10)}$\\
\hline
$(-2,-2)$ & $\F_{(-12)}$\\
\hline
$(-3,k)$ and $(k,-3)$, $k\ge -3$ & $\F_{(2k-8)}\oplus \F_{(2k-9)}$\\
\hline
$(-4,k)$ and $(k,-4)$, $k\ge 10$ &  $\F_{(2k-10)}\oplus \F_{(2k-11)}$\\
\hline
$(-5,-5)$ & $\F_{(-22)}$ \\
\hline
$(a,b)$, $a,b\le -6$ &  $\F_{(2a+2b)}\oplus \F_{(2a+2b-1)}$\\
\hline
\end{tabular}
\caption{Maslov gradings for the (4,6) cable of the trefoil}
\label{tab:4 6 cable of trefoil}
\end{table}

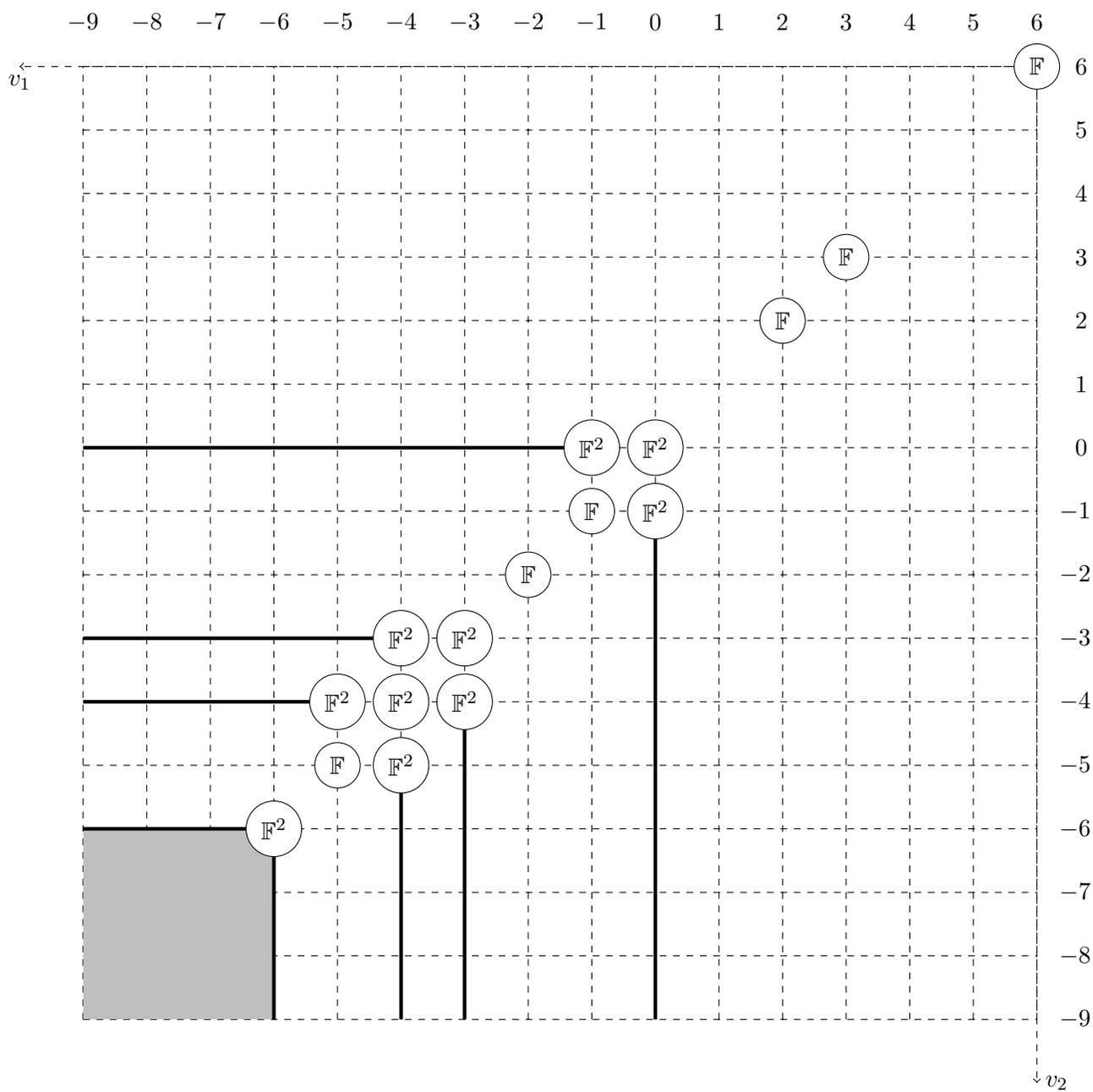
\begin{figure}[!ht]
\begin{tikzpicture}
\draw[dashed,->] (0,0)--(-16,0) node [below] {$v_1$};
\draw[dashed,->] (0,0)--(0,-16) node [right] {$v_2$};
 
\draw [step=1,dashed] (-15,-15) grid (0,0);

\fill[fill=lightgray] (-12,-12)--(-12,-15)--(-15,-15)--(-15,-12)--(-12,-12);
\draw[ultra thick] (-15,-12)--(-12,-12)--(-12,-15);

\draw[ultra thick] (-15,-6)--(-7,-6);
\draw[ultra thick] (-6,-7)--(-6,-15);
\draw[ultra thick] (-15,-9)--(-10,-9);
\draw[ultra thick] (-9,-10)--(-9,-15);
\draw[ultra thick] (-15,-10)--(-11,-10);
\draw[ultra thick] (-10,-11)--(-10,-15);

\draw (0,0) node [circle,draw=black,fill=white] {$\F$};
\draw (-3,-3) node [circle,draw=black,fill=white] {$\F$};
\draw (-4,-4) node [circle,draw=black,fill=white] {$\F$};
\draw (-6,-6) node [circle,draw=black,fill=white] {$\F^2$};
\draw (-7,-7) node [circle,draw=black,fill=white] {$\F$};
\draw (-8,-8) node [circle,draw=black,fill=white] {$\F$};
\draw (-9,-9) node [circle,draw=black,fill=white] {$\F^2$};
\draw (-10,-10) node [circle,draw=black,fill=white] {$\F^2$};
\draw (-11,-11) node [circle,draw=black,fill=white] {$\F$};
\draw (-12,-12) node [circle,draw=black,fill=white] {$\F^2$};

\draw  (-6,-7) node [circle,draw=black,fill=white] {$\F^2$};
\draw  (-7,-6) node [circle,draw=black,fill=white] {$\F^2$};
\draw  (-9,-10) node [circle,draw=black,fill=white] {$\F^2$};
\draw  (-10,-9) node [circle,draw=black,fill=white] {$\F^2$};
\draw  (-10,-11) node [circle,draw=black,fill=white] {$\F^2$};
\draw  (-11,-10) node [circle,draw=black,fill=white] {$\F^2$};
 
\draw (0,0.7) node {$6$}; 
\draw (-1,0.7) node {$5$}; 
\draw (-2,0.7) node {$4$}; 
\draw (-3,0.7) node {$3$}; 
\draw (-4,0.7) node {$2$}; 
\draw (-5,0.7) node {$1$}; 
\draw (-6,0.7) node {$0$}; 
\draw (-7,0.7) node {$-1$}; 
\draw (-8,0.7) node {$-2$}; 
\draw (-9,0.7) node {$-3$};   
\draw (-10,0.7) node {$-4$}; 
\draw (-11,0.7) node {$-5$}; 
\draw (-12,0.7) node {$-6$}; 
\draw (-13,0.7) node {$-7$};    
\draw (-14,0.7) node {$-8$}; 
\draw (-15,0.7) node {$-9$};    
 
\draw (0.7,0) node {$6$}; 
\draw (0.7,-1) node {$5$}; 
\draw (0.7,-2) node {$4$}; 
\draw (0.7,-3) node {$3$}; 
\draw (0.7,-4) node {$2$}; 
\draw (0.7,-5) node {$1$}; 
\draw (0.7,-6) node {$0$}; 
\draw (0.6,-7) node {$-1$}; 
\draw (0.6,-8) node {$-2$}; 
\draw (0.6,-9) node {$-3$};  
\draw (0.6,-10) node {$-4$}; 
\draw (0.6,-11) node {$-5$}; 
\draw (0.6,-12) node {$-6$}; 
\draw (0.6,-13) node {$-7$}; 
\draw (0.6,-14) node {$-8$}; 
\draw (0.6,-15) node {$-9$}; 
 
\end{tikzpicture}
\caption{$\HFL$ for  the (4,6) cable of the trefoil: $\F^2$  on thick lines and in the grey region}
\label{fig:4 6 cable of trefoil}
\end{figure}

\end{document}